\documentclass[12pt]{amsart}

\usepackage{amssymb}
\numberwithin{equation}{section}

\usepackage{lineno,hyperref,afterpage}
\usepackage{amsmath,amssymb}
\usepackage{xcolor}
\usepackage{amsthm} 
\usepackage[shortlabels]{enumitem}
\usepackage[scale=0.7]{geometry} 
\usepackage{verbatim}

\usepackage{mathtools}

\DeclarePairedDelimiter\ceil{\lceil}{\rceil}
\DeclarePairedDelimiter\floor{\lfloor}{\rfloor}

\allowdisplaybreaks
\modulolinenumbers[5]

\newtheorem{The}{Theorem}[section]
\newtheorem{Lem}[The]{Lemma}
\newtheorem{Cor}[The]{Corollary}
\newtheorem{Rem}[The]{Remark} 

\newtheorem{Pro}[The]{Proposition}

\newcommand{\R}{{\mathbb R}}

\newcommand{\N}{{\mathbb N}}

\newcommand{\eps}{\varepsilon}

\newcommand{\p}{\partial}

\newcommand{\norm}[1]{\left\lVert #1 \right\rVert}
\newcommand{\abs}[1]{\left\lvert #1 \right\rvert}

\allowdisplaybreaks

\usepackage[numbers]{natbib}
\begin{document}

\title[Geometrical optics for the fractional Helmholtz equation]{Geometrical optics for the fractional Helmholtz equation and applications to inverse problems}

\author{Giovanni Covi}
\address{Department of Mathematics and Statistics, University of Helsinki, Finland}
\email{\texttt{giovanni.covi@helsinki.fi}} 

\author{Maarten de Hoop}
\address{Department of Computational and Applied Mathematics, Rice University, Houston, TX, USA}
\email{\texttt{mvd2@rice.edu}}

\author{Mikko Salo}
\address{Department of Mathematics and Statistics, University of Jyv\"askyl\"a, Finland}
\email{\texttt{mikko.j.salo@jyu.fi}}

\begin{abstract}
In this paper we construct a parametrix for the fractional Helmholtz equation $((-\Delta)^s - \tau^{2s} r(x)^{2s} + q(x))u=0$ making use of geometrical optics solutions. We show that the associated eikonal equation is the same as in the classical case, while in the first transport equation the effect of nonlocality is only  {visible} in the zero-th order term, which depends on $s$. Moreover, we show that the  {approximate} geometrical optics solutions present different behaviors in the regimes $s\in(0,\frac 12)$ and $s\in [\frac 12,1)$. While the latter case is quite similar to the classical one, which corresponds to $s=1$, in the former case we find that the potential is a strong perturbation, which changes the propagation of singularities. As an application, we study the inverse problem consisting in recovering the potential $q$ from Cauchy data when the refraction index $r$ is fixed and simple. Using our  {parametrix based on the} construction of  {approximate} geometrical optics solutions, we prove that H\"older stability holds for this problem. This is a substantial improvement over the state of the art for fractional wave equations, for which the usual Runge approximation argument can provide only logarithmic stability. Besides its mathematical novelty, this study is motivated by envisioned applications in nonlocal elasticity models emerging from the geophysical sciences.
\end{abstract}

\maketitle
\tableofcontents

\section{Introduction}

In this article we study the fractional Helmholtz operator  
\[
L = ((-\Delta)^s - \tau^{2s} r(x)^{2s} + q(x))
\]
and the fractional Helmholtz equation 
\[
Lu = 0 \text{ in $\Omega$}.
\]
As we will explain below, this is a scalar model for equations appearing in nonlocal elasticity. We assume that $\Omega \subset \R^n$ is a bounded domain with smooth boundary, $r \in C^{\infty}(\R^n)$ is a positive function (index of refraction), $q \in C^{\infty}(\R^n)$ is a potential, and $\tau > 0$ is a (large) frequency. The operator $(-\Delta)^s$ is the fractional Laplacian, defined via the Fourier multiplier $|\xi|^{2s}$ where $0 < s < 1$. We consider functions $u \in H^s(\R^n)$ that solve $Lu = 0$ in $\Omega$ in the sense of distributions.

In the classical case where $s=1$, the above equation is the standard Helmholtz equation. If the index of refraction $r(x)$ is sufficiently nice (e.g.\ the Riemannian metric $g_{jk}(x) = r(x)^{2} \delta_{jk}$ in $\overline{\Omega}$ is simple, meaning that it is simply connected with strictly convex boundary and no conjugate points \cite{PSU23}), then the equation $(-\Delta - \tau^2 r^2 + q) u = 0$ admits geometrical optics solutions that concentrate along geodesics of the metric $g$ when $\tau \to \infty$. These solutions have the form 
\[
u = e^{i\tau \varphi} (a_0 + \tau^{-1} a_1 + \tau^{-2} a_2 + \ldots)
\]
where the phase function $\varphi \in C^{\infty}(\overline{\Omega})$ is a real valued solution of the eikonal equation 
\[
|\nabla \varphi| = r,
\]
and the principal amplitude $a_0 \in C^{\infty}(\overline{\Omega})$ solves the transport equation 
\[
2 \nabla \varphi \cdot \nabla a_0 + (\Delta \varphi) a_0 = 0.
\]
Moreover, if we write $h = 1/\tau$ and let $P = h^2 r^{-2} L = -h^2 r^{-2} \Delta - 1 + h^2 r^{-2} q$, then semiclassical singularities of solutions of $Pu = 0$ propagate along the bicharacteristics of the semiclassical operator $P$ \cite[Section 12.3]{Zworski_Semiclassical_Analysis}. The spatial projections of these bicharacteristics are the geodesics of the metric $g$.

We wish to obtain analogous results in the nonlocal case where $0 < s < 1$. Our first result shows that for $1/2 \leq s < 1$ the fractional Helmholtz equation, even though it is nonlocal, admits geometrical optics solutions that are very similar to the classical case $s=1$.

\begin{The}[Case $1/2 \leq s < 1$] \label{thm_main1}
Suppose that $\varphi \in C^{\infty}(\R^n)$ solves the eikonal equation 
\[
|\nabla \varphi| = r \text{ in $\Omega$},
\]
and $a_0 \in C^{\infty}_c(\R^n)$ solves the transport equation 
\[
\left. \begin{array}{rl} 
2 \nabla \varphi \cdot \nabla a_0 + b_s a_0 = 0 \text{ in $\Omega$}, & \text{if $1/2 < s < 1$,} \\[5pt]
2 \nabla \varphi \cdot \nabla a_0 + b_s a_0 = q \text{ in $\Omega$}, & \text{if $s=1/2$,}
\end{array} \right.
\]
where $b_s := \Delta \varphi + (2s-2)\frac{\varphi'' \nabla \varphi \cdot \nabla \varphi}{\abs{\nabla \varphi}^2}$.

Then for any $N\in\mathbb N$ there is a solution of $Lu = 0$ in $\Omega$ having the form 
\[
u =  e^{i\tau \varphi} (a_0 + \tau^{-\alpha_1} a_1 + \ldots + \tau^{-\alpha_N} a_N) + R_N
\]
where $a_1, \ldots, a_N \in C^{\infty}_c(\R^n)$ are independent of $\tau$, and $R_N \in H^s(\R^n)$ satisfies  
\[
\norm{R_N}_{H^s} = O(\tau^{-C_sN})
\]
for a fixed constant $C_s>0$ depending only on $s$. Above, $(\alpha_j)_{j=1}^{\infty} \subset \R_+$ is a strictly increasing sequence depending on $s$ such that  $\alpha_j \to \infty$ as $j \to \infty$.
\end{The}

We make the following remarks:

\begin{itemize}
\item 
The eikonal equation for $1/2 \leq s < 1$ is exactly the same as for $s=1$, which indicates that singularities propagate along geodesics.
\item 
The first transport equation for $1/2 < s < 1$ is almost the same as for $s=1$. The effect of nonlocality is only visible in the zeroth order term $b_s$ that depends on $s$. Thus $a_0$ will contain an exponential $s$-dependent attenuation factor.
\item 
For $s=1/2$ the first transport equation also contains the potential $q$, which can be considered as part of the subprincipal symbol (the principal part is of order $1$ and $q$ is of order $0$). This is analogous to the case of first order perturbations of the Laplacian.
\item 
The functions $\varphi$ and $a_j$ need to be defined on $\R^n$ to have $u \in H^s(\R^n)$. Their values outside $\overline{\Omega}$ will affect $R_N$, but not the asymptotics in $\tau$.
\end{itemize}

We also note that for $1/2 \leq s < 1$, one could formally apply the semiclassical propagation of singularities theorem in \cite[Section 12.3]{Zworski_Semiclassical_Analysis} to conclude that singularities propagate along geodesics. This is not quite rigorous since the symbol of $h^{2s} (-\Delta)^{s}$ is not smooth at $\xi=0$. Our proof of Theorem \ref{thm_main1} is based on stationary phase expansions that deal with this nonsmoothness.

Next we state a result for $0 < s < 1/2$. In this case the potential $q$ is a strong perturbation (the principal part has order $2s < 1$ and $q$ is of order $0$), which changes the propagation of singularities. Similar phenomena appear in long range scattering theory \cite[Volume 4]{HO:analysis-of-pdos}. In the geometrical optics construction this will be visible in the presence of lower order phase functions that give rise to exponential factors  oscillating at different frequencies.

\begin{The}[Case $0 < s < 1/2$] \label{thm_main2}
Suppose that $\varphi_0 \in C^{\infty}(\R^n)$ solves the eikonal equation 
\[
|\nabla \varphi_0| = r \text{ in $\Omega$},
\]
and $a_0 \in C^{\infty}_c(\R^n)$ solves the transport equation 
\[
2 \nabla \varphi \cdot \nabla a_0 + b_s a_0 = 0 \text{ in $\Omega$}
\]
where $b_s := \Delta \varphi + (2s-2)\frac{\varphi'' \nabla \varphi \cdot \nabla \varphi}{\abs{\nabla \varphi}^2}$.

Then for any $N\in\mathbb N$ there exist $M\in\mathbb N$ and a solution of $Lu = 0$ in $\Omega$ having the form 
\[
u =  e^{i\tau (\varphi_0 + \tau^{-2s} \varphi_1 + \ldots + \tau^{-2 M s} \varphi_M)} (a_0 + \tau^{-\alpha_1} a_1 + \ldots + \tau^{-\alpha_M} a_M) + R_N
\]
where $\varphi_j \in C^{\infty}(\R^n)$ and $a_j \in C^{\infty}_c(\R^n)$ are independent of $\tau$, and $R_N \in H^s(\R^n)$ satisfies  
\[
\norm{R_N}_{H^s} = O(\tau^{-C_sN})
\]
for a fixed constant $C_s>0$ depending only on $s$.
\end{The}

Note that in both Theorem \ref{thm_main1} and \ref{thm_main2}, if $x_0 \notin \overline{\Omega}$ and if $(\rho,\theta)$ are Riemannian normal coordinates for the metric $g_{jk} = r(x)^2 \delta_{jk}$ centered at $x_0$, then choosing $\varphi(x) = \varphi_0(x) = \mathrm{dist}_g(x,x_0) = \rho$ and $a_0(\rho,\theta) = e^{F(\rho,\theta)} b(\theta)$, where $F$ is a certain fixed function depending on $r$ and $q$, while $b$ is any function supported near $\theta_0$, yields solutions that concentrate near the geodesic $\rho \mapsto (\rho, \theta_0)$. See the proof of Theorem \ref{Th:stability-high-s} for more details.

As an application of the geometrical optics construction, we consider the inverse problem consisting in recovering the potential $q$ from Cauchy data in the case of fixed, but unknown, refraction index $r$. In particular, we have the following Hölder stability result.

\begin{The}\label{Th:stability-high-s}
    Let $s\in [\frac12,1)$, $\Omega\subseteq \R^n$ be open and bounded. Assume that the refraction index $r \in C^{\infty}(\R^n)$ is positive with $r\equiv 1$ in $\mathbb R^n\setminus\Omega$, and that the manifold $(\overline{\Omega}, g)$ with $g_{jk} = r(x)^2 \delta_{jk}$ is simple. Moreover, let $k,N\in\N$, and assume that $q_1, q_2 \in H^k(\R^n)$ are such that $\|q_1\|_{H^k}, \|q_2\|_{H^k} \leq N$, with $q_1=q_2$ in $\mathbb R^n\setminus\Omega$. Then there exist constants $\tau_0, \gamma, C>0$, with $\gamma$ depending only on $s$, such that
    $$\|q_1-q_2\|_{L^2(\Omega)} \leq C \left[ \sup_{\tau \geq \tau_0} \delta(C_{r,q_1}^{\tau}, C_{r,q_2}^{\tau}) \right]^\gamma.$$
\end{The}

The symbol $C_{r,q_j}^{\tau}$ indicates the Cauchy data set associated to the operator $$P_j := (-\Delta)^s - \tau^{2s} r^{2s} + q_j,$$ while the function $\delta$ measures the distance between the Cauchy data sets. In the case that $\Omega$ is Lipschitz and the exterior Dirichlet-to-Neumann (DN) map $\Lambda_{P_j}$ is well defined, we have
$$C_{r,q_j}^{\tau}:=\{ (f, \Lambda_{P_j} f) \in H^s(\Omega_e)\times H^{-s}(\Omega_e)\}.$$
In the same case, we can define the distance $\delta$ as
\begin{align*}
    \delta(C_{r,q_1}^{\tau},C_{r,q_2}^{\tau}) :&= \max_{j\neq k} \sup_{\substack{
     f_1 \in H^s(\Omega_e): \\ \|f_1\|=1} 
} \inf_{ f_2\in H^s(\Omega_e)} \left(\| f_1-f_2 \|_{H^s(\Omega_e)} + \| \Lambda_{P_j} f_1 - \Lambda_{P_k} f_2 \|_{H^{-s}(\Omega_e)} \right) .
\end{align*} 
We will rigorously define both the Cauchy data and the distance function $\delta$ in greater generality in Section \ref{sec:applications}. Note that existing stability results for fixed frequency have a logarithmic modulus \cite{RS17a} and this is optimal \cite{RS-expo}. The fact that we use many frequencies leads to the improved H\"older stability result in Theorem \ref{Th:stability-high-s}.

Because of the particular nature of the geometrical optics solutions in the regime $s\in (0,\frac 12)$, which differs substantially from the classical case $s=1$, we can apply the method in Theorem \ref{Th:stability-high-s} only in the regime $s\in[\frac12,1)$. A precise account of this observation is given in Remark \ref{rem:why-not}. 

\subsection{Motivation}\label{sec-motivation}
The main motivation of our study on the wave properties of fractional Helmholtz equations comes from the field of nonlocal elasticity. 

Local (or classical) continuum theories of elasticity, while in perfect agreement with experiments involving only measurements at scales which are large in comparison to the intermolecular distances of the elastic material, are unable to predict physical phenomena at the smallest scales, where the intrinsic granularity of the examined material becomes important. As observed by Eringen in \cite{Eri84} (see also \cite{BenZion} and \cite{Failla-Zingales-2020} for a recent overview), one of the most critical failures of the local theory can be identified in the predicted value of the stress field at a sharp crack tip, which produces a singularity, clearly contradicting experimental measurements. The local theory also wrongly predicts that phase velocities should be independent of wavelength. This problematic effect is seen also in the acoustic (or scalar) case. As discussed in \cite{Petersdorff-Stephan}, the solutions of the Dirichlet problem for the Laplacian in a domain with a polyhedral boundary present so-called \emph{corner singularities}, and as such do not have the regularity which would be expected of solutions of elliptic equations in regular domains. When the domain is an infinite wedge, a polyhedral cone, or a polygon in the 2D case, the solutions exhibit a similar behaviour.

While many strategies have been proposed to correct these problematic observations, the most elegant is the one by Eringen, Speziale and Kim (\cite{ESK77}), which does not need to modify the ordinary hypothesis of rupture based on the maximum stress hypothesis as in the work by Griffith and collaborators (see e.g. \cite{Griffith1920}). By incorporating elements of the lattice structure of the elastic material into the model, the theory of \emph{nonlocal elasticity} introduced in \cite{ESK77} produces predictions in complete accordance with experimental data in the sharp crack tip case, and reduces to the classical theory in the limit of vanishing atomic scale. 

Since the publication of the cited seminal papers by Eringen \cite{Eringen-1972, Er02}, the theory of nonlocal elasticity has been widely adopted by the scientific community and extended in various directions. A variational formulation of nonlocal elasticity has been the object of the works of Mindlin \cite{Mindlin-1964, Mindlin-1965}. Time-fractional wave equations have been used by Zorica and Oparnica \cite{ZoricaOparnica2020} for modeling viscoelastic behavior. Nonlocal elasticity was also studied in relation to inverse problems by Askes and Aifantis \cite{AskesAifantis}. Other examples include the recent studies \cite{WS22, WL22}, which have considered aspects of the theory concerning the relation between elastic stress and the equilibrium corner angle, as well as rigid inclusions in nonlocal elastic materials. The overview by Jin and Rundell \cite{JinRundell2015} contains a multitude of references on fractional derivative equations and related inverse problems.

In our work \cite{CdHS24}, we considered a model of nonlocal elasticity related to the fractional linear gradient-elasticity model in \cite{CarpinteriCornettiSapora-2011, TarasovAifantis-2018}, in which the nonlocal stress is defined as the Riesz fractional integral of the strain field. In that occasion, we have proved uniqueness in the inverse problem of determining the Lam\'e parameters of the isotropic material from nonlocal DN data, under the additional assumption that the Poisson ratio is fixed a priori. This has shown the intrinsic limits of the nonlocal technique we used to obtain such result. In this paper, we shall generalize the well known geometrical optics technique, which is classically available for the \emph{local} case, to the inverse problem for the \emph{fractional} Helmholtz equation. We understand the present article, which considers the acoustic (scalar) case, as a first step towards the wider goal of studying geometrical optics for the fractional elastic (vector) operator we introduced in \cite{CdHS24}. This will be the object of future work.

\subsection{Connection to the literature}
The fractional Calderón problem is a prototypical inverse problem for a nonlocal operator, which in recent years has attracted the attention of many authors. Introduced in the seminal paper \cite{GSU16}, it asks to recover the potential $q$ in the fractional Schr\"odinger equation
$$\begin{cases}
    (-\Delta)^su+qu =0, & \Omega \\ u =f, & \Omega_e:=\mathbb R^n\setminus\overline\Omega
\end{cases}$$
from partial data given in the form of a nonlocal Dirichlet-to-Neumann (DN) map $\Lambda_q$, which is defined as
$$ \Lambda_q : H^s(\Omega_e)\ni f \mapsto (-\Delta)^su|_{\Omega_e}\in H^s(\Omega_e)^*. $$
Here $s\in (0,1)$, $n\in \mathbb N$, $\Omega$ is a bounded open set, and $(-\Delta)^s$ is the fractional Laplace operator, which can be defined via the Fourier transform $\mathcal F$ as 
$$(-\Delta)^su = \mathcal F^{-1}(|\xi|^{2s}\mathcal F u(\xi)),$$
or equivalently as the singular integral 
$$ (-\Delta)^su(x) := c_{n,s}PV\int_{\mathbb R^n}\frac{u(x)-u(y)}{|x-y|^{n+2s}}dy, $$
for a fixed constant $c_{n,s}:=\frac{4^s\Gamma(n/2+s)}{\pi^{n/2}|\Gamma(-s)|}$. The main technique available for the solution of this inverse problem was already given in \cite{GSU16}, where it was shown that a potential $q\in L^\infty$ can be uniquely recovered. The method was based on the strong unique continuation property of the fractional Laplacian, which can be shown to be equivalent to the Runge approximation property, and on the Alessandrini identity. 

The uniqueness result was later improved in \cite{RS17b} to potentials in $q\in L^{n/2s}$. In the paper \cite{GRSU18}, the authors further showed uniqueness and reconstruction even in the case of a single measurement. The stability and instability properties of the fractional Calder\'on problem were also explored. The works \cite{RS17a, RS17b, R20} proved exponential stability, in accordance to the results known in the classical case. In the case of finite dimension, Lipschitz stability was obtained in \cite{RS18}. As for instability, the general technique of \cite{KRS21} based on the study of entropy numbers was recently applied to the fractional Calder\'on problem in \cite{BCR24}. The fractional Schrödinger equation has been studied in many variations, including with general local (\cite{CMRU22,CLR18, CMR21}), magnetic (\cite{C20, Li20a, Li20b}) and quasilocal perturbations (\cite{C21}). The case of closed and complete Riemannian manifolds was also considered (\cite{FKU24,CO23}). Many more related nonlocal operators were the objects of study of \cite{BS22, CGR22, GLX17, KLW21, LLR19}, among others.

The conductivity formulation of the fractional Calder\'on problem was introduced in \cite{C20a}, making use of the nonlocal vector calculus of \cite{DGLZ12, DGLZ13}. For this formulation, one considers the fractional gradient $\nabla^s$ defined as
$$ \nabla^su(x,y):= \frac{c^{1/2}_{n,s}}{\sqrt{2}}\frac{u(y)-u(x)}{|x-y|^{n/2+s+1}}(x-y),$$
using which one can pose the fractional conductivity equation
$$\begin{cases}
    \mathbf C^s_\gamma u := (\nabla\cdot)^s(\gamma^{1/2}(x)\gamma^{1/2}(y)\nabla^s u) =0, & \Omega \\ u =f, & \Omega_e
\end{cases}.$$
Here $\gamma$ is a scalar conductivity function, while $(\nabla\cdot)^s := (\nabla^s)^*$ is the fractional divergence, i.e. the adjoint of the fractional gradient. Observe that $(\nabla\cdot)^s\nabla^s = (-\Delta)^s$. The inverse problem for the fractional conductivity equation consists in finding $\gamma$ from a nonlocal DN map. It was shown in \cite{C20a} to be solvable by means of the fractional Liouville reduction, a transformation which reformulates the problem in fractional Schr\"odinger form. The paper \cite{C20a} also shows how the fractional conductivity operator naturally arises as a limit of a random walk with long jumps on a square lattice (see also \cite{Val09}). This interpretation is extended to a general graph $G$ in the recent work \cite{CL24}, where the inverse problem consisting in recovering both the edge set of $G$ and the conductivity $\gamma$ from partial random walk data is solved up to natural constraints by means of a new algebraic method. In \cite{RZ22a} the results of \cite{C20a} are improved to conductivities in $H^{s,n/s}$. This inverse problem was also considered in the higher order regime (\cite{CMR21}), anisotropic (\cite{Cov22}) and global cases (\cite{RZ22b,RZ22c}). Stability and instability results for the fractional conductivity formulation are the objects of \cite{CRTZ22,BCR24}. 

 In particular, a fractional isotropic elasticity operator was studied by the authors of this paper in \cite{CdHS24}, where uniqueness was proved to hold under strong assumptions on the Poisson ratio. Improving the results of this study with new powerful techniques borrowed from the theory of inverse problems for local operators is one of the main motivations of this paper (see Section \ref{sec-motivation}).
 
 We refer to the surveys \cite{S17,C24} for more references on the fractional Calderón problem in both the Schrödinger and conductivity formulations. 

\subsection{Organization of the article}

The rest of the article is organized as follows. Section \ref{Sec:approx} is dedicated to the construction of approximate geometrical optics solutions for the fractional Helmholtz equation. We first consider the case of constant coefficients, and then prove the result for variable refraction index and potential. In Section \ref{sec:exact} we employ a resolvent estimate to upgrade the approximate geometrical optics solutions from the previous section to exact solutions, thus proving Theorems \ref{thm_main1} and \ref{thm_main2}. Applications to inverse problems are considered in Section \ref{sec:applications}, where the stability result in Theorem \ref{Th:stability-high-s} is proved. Finally, in the Appendix we show how to compute the transport equation for the principal amplitude of the geometrical optics solutions in polar normal coordinates.

\section{Geometrical optics solutions for the fractional Helmholtz equation}\label{Sec:approx}

In this section we construct geometrical optics solutions for different fractional equations. We start in section \ref{subsec-helm-fix} with the fractional Helmholtz equation with constant coefficients, for which we construct geometrical optics solutions directly by using the Fourier transform. Next, we consider the construction of approximate geometrical optics solutions for the case of variable refraction index and potential (section \ref{subsec-helm-rq}). In the coming section \ref{sec:exact} we shall upgrade these solutions from approximate to exact.

\subsection{The case of constant coefficients} \label{subsec-helm-fix}

As a warmup, we will construct an approximate geometrical optics solution to the equation 
\begin{equation*}
    ((-\Delta)^s - \tau^{2s}) u = 0 \text{ in $\Omega \subset \R^n$}.
\end{equation*}

\begin{Pro}\label{prop:exact-sol-helm}
    For any $M\in\N$ there exists a positive constant $C_M$ and amplitude functions $\{a_l\}_{l=0}^{M}$ such that the function $$u(x):= e^{i\tau\alpha\cdot x} a(x), \qquad\mbox{with}\qquad a(x): = \sum_{l=0}^{M} \tau^{-l}a_l(x)$$ and $\alpha\in \mathbb S^{n-1}$, verifies \begin{equation*}
    \|((-\Delta)^s - \tau^{2s}) u\|_{L^2(\Omega)} \leq C_M\tau^{-M}.
\end{equation*}
Moreover, the principal amplitude $a_0$ solves the transport equation $\alpha\cdot\nabla a_0 =0$.
\end{Pro}

\begin{proof}
We start by computing 
\begin{align*}
    (-\Delta)^su & = \frac{1}{(2\pi)^n}\int_{\R^n} e^{ix\cdot\xi}|\xi|^{2s}\mathcal F(e^{i\tau\alpha\cdot x}a)d\xi
    \\ &
    = \frac{1}{(2\pi)^n}\int_{\R^n} e^{ix\cdot\xi}|\xi|^{2s}\hat a(\xi-\tau\alpha)d\xi 
    \\ &
    = \frac{e^{i\tau\alpha\cdot x}}{(2\pi)^n}\int_{\R^n} e^{ix\cdot\xi}|\xi+\tau\alpha|^{2s}\hat a(\xi)d\xi.
\end{align*}
Assume that $\alpha\cdot\xi \neq 0$. Then using the formula for the binomial series $(1+x)^s = \sum_{j=0}^\infty \binom{s}{j} x^{j}$ we compute

\begin{align*}
    |\xi+\tau\alpha|^{2s} & = (|\xi|^2 + \tau^2 + 2\tau\alpha\cdot\xi)^s
    \\ & 
    = \tau^{2s}\left(1+2\tau^{-1}(\alpha\cdot\xi)\left(  1+  2^{-1}\tau^{-1}|\xi|^{2}(\alpha\cdot\xi)^{-1}     \right)\right)^s
    \\ & 
    = \tau^{2s}\sum_{j=0}^\infty \binom{s}{j} 2^j\tau^{-j}(\alpha\cdot\xi)^j\left(  1+  2^{-1}\tau^{-1}|\xi|^{2}(\alpha\cdot\xi)^{-1}     \right)^j 
    \\ & 
    = \tau^{2s} + \tau^{2s}\sum_{j=1}^\infty \binom{s}{j} 2^j\tau^{-j}(\alpha\cdot\xi)^j\left(  1+  2^{-1}\tau^{-1}|\xi|^{2}(\alpha\cdot\xi)^{-1}     \right)^j
    \\ & 
    = \tau^{2s} + \tau^{2s}\sum_{j=1}^\infty \binom{s}{j} 2^j\tau^{-j}(\alpha\cdot\xi)^j\sum_{k=0}^j \binom{j}{k} 2^{-k}\tau^{-k}|\xi|^{2k}(\alpha\cdot\xi)^{-k}
    \\ &  
    = \tau^{2s} + \sum_{j=1}^\infty \sum_{k=0}^j  c'_{s,j,k} \tau^{2s-j-k}|\xi|^{2k}(\alpha\cdot\xi)^{j-k},
\end{align*}
where $c'_{s,j,k}:=\binom{s}{j}\binom{j}{k} 2^{j-k}$. By the assumption $ a(x) = \sum_{l=0}^{M} a_l(x)\tau^{-l}$, we have $ \hat a(\xi) = \sum_{l=0}^{M} \hat a_l(\xi)\tau^{-l}$. Then, because for all $\alpha\in\mathbb S^{n-1}$ the set of $\xi\in\mathbb R^n$ orthogonal to $\alpha$ has vanishing measure, we have

\begin{align*}
    (-\Delta)^su & = \frac{e^{i\tau\alpha\cdot x}}{(2\pi)^n}\int_{\R^n} e^{ix\cdot\xi}\tau^{2s}\hat a(\xi)d\xi 
    \\ & \quad 
    +\frac{e^{i\tau\alpha\cdot x}}{(2\pi)^n}\int_{\R^n} e^{ix\cdot\xi}\left( \sum_{j=1}^\infty \sum_{k=0}^j  c'_{s,j,k} \tau^{2s-j-k}|\xi|^{2k}(\alpha\cdot\xi)^{j-k}\right)\left( \sum_{l=0}^{M}  \hat a_l(\xi)\tau^{-l} \right)d\xi
    \\ & 
    = e^{i\tau\alpha\cdot x}\tau^{2s}a + \frac{e^{i\tau\alpha\cdot x}}{(2\pi)^n} \sum_{j=1}^\infty \sum_{k=0}^j \sum_{l=0}^{M}  c'_{s,j,k}\tau^{2s-(j+k+l)} \int_{\R^n} e^{ix\cdot\xi}  |\xi|^{2k}(\alpha\cdot\xi)^{j-k} \hat a_l(\xi) d\xi
    \\ & 
    = \tau^{2s}u +  e^{i\tau\alpha\cdot x} \sum_{j=1}^\infty \sum_{k=0}^j \sum_{l=0}^{M}  c'_{s,j,k}\tau^{2s-(j+k+l)} \mathcal{F}^{-1}\left[  |\xi|^{2k}(\alpha\cdot\xi)^{j-k} \hat a_l(\xi) \right](x)
    \\ & 
    = \tau^{2s}u +  e^{i\tau\alpha\cdot x} \sum_{\nu=1}^\infty \tau^{2s-\nu} \sum_{l=0}^{\min\{\nu-1,M\}} \mathcal{F}^{-1}\{  \psi_{\nu,l}(\xi)\hat a_l(\xi) \}(x)
    \\ & 
    = \tau^{2s}u +  e^{i\tau\alpha\cdot x} \sum_{\nu=1}^\infty \tau^{2s-\nu} \sum_{l=0}^{\min\{\nu-1,M\}} \Psi_{\nu,l}a_l,
\end{align*}
where $ \Psi_{\nu,l} $ is the operator whose symbol $ \psi_{\nu,l} $ is given by $$ \psi_{\nu,l}(\xi):= \sum_{\substack{j\in\N_+, k \in \{0,...,j\} : \\ \nu-l=j+k}} c'_{s,j,k}  |\xi|^{2k}(\alpha\cdot\xi)^{j-k}.$$
Observe that $\psi_{l+1,l}(\xi)=2s\alpha\cdot\xi$ for all $l\in\N$. 
We assume that the coefficient of $\tau^{2s-1}$ vanishes, that is
$$ 0= \mathcal{F}^{-1}\{  \psi_{1,0}(\xi)\hat a_0(\xi) \} = 2s\alpha\cdot\mathcal{F}^{-1}\{\xi\hat a_0(\xi) \} = -2is\alpha\cdot\nabla a_0 \quad\Rightarrow\quad \alpha\cdot\nabla a_0 =0. $$
Thus we see that $a_0$ solves the time-independent transport equation. Now we have
\begin{align*}
(-\Delta)^su-\tau^{2s}u & = e^{i\tau\alpha\cdot x} \sum_{\nu=2}^\infty \tau^{2s-\nu} \sum_{l=0}^{\min\{\nu-1,M\}} \Psi_{\nu,l}a_l  
\\ & 
= e^{i\tau\alpha\cdot x} \tau^{2s-2}\sum_{\nu=0}^\infty \tau^{-\nu} \sum_{l=0}^{\min\{\nu+1,M\}} \Psi_{\nu+2,l}a_l
\\ & 
= e^{i\tau\alpha\cdot x} \tau^{2s-2} \left(\sum_{\nu=0}^{M-1} \tau^{-\nu} \sum_{l=0}^{\nu+1} \Psi_{\nu+2,l}a_l + \sum_{\nu\geq M} \tau^{-\nu} \sum_{l=0}^{M} \Psi_{\nu+2,l}a_l\right)
\\ & 
= e^{i\tau\alpha\cdot x} \tau^{2s-2}\sum_{\nu=0}^{M-1} \tau^{-\nu} \left( \Psi_{\nu+2,\nu+1}a_{\nu+1} + \sum_{l=0}^{\nu} \Psi_{\nu+2,l}a_l\right) + O(\tau^{-M}).
\end{align*}
Requiring that the terms in parenthesis vanish, for all $\nu < M$ we get
$$\alpha\cdot\nabla a_{\nu+1}=\frac i{2s}\Psi_{\nu+2,\nu+1}a_{\nu+1} = -\frac i{2s}\sum_{l=0}^{\nu} \Psi_{\nu+2,l}a_l,$$
that is $a_{\nu+1}$ solves an inhomogeneous time-independent transport equation, in which the right hand side is computed recursively. This gives the required estimate and completes the proof.
\end{proof}

\subsection{The case of variable refraction index and potential}\label{subsec-helm-rq}

In this section we will construct an approximate geometrical optics solution to the equation 
\begin{equation}\label{rq-equation}
    ((-\Delta)^s - \tau^{2s} r(x)^{2s} + q(x)) u = 0 \text{ in $\Omega \subset \R^n$},
\end{equation}
where $r(x), q(x)$ are smooth positive functions (index of refraction and potential). 
\begin{Pro}\label{prop:approx-sol-rq}
    For all $M,\beta\in\N$ there exist two constants $C_s, C_{M,\beta}>0$ such that $$u_M(x):= e^{i\tau\varphi(x)} a(x)$$ verifies $$\|((-\Delta)^s-\tau^{2s}r(x)^{2s}+q(x))u_M(x)\|_{H^\beta_{scl}(\Omega)}\leq C_{M,\beta} \tau^{-M},$$
    where the phase and amplitude functions $\varphi,a$ are given by $$\varphi(x):= \sum_{j=0}^{F_M} \tau^{-2s j}\varphi_j(x), \qquad a(x):= \sum_{l=0}^{A_M} \tau^{-\alpha_l} a_{l}(x)$$ with $A_M = F_M := C_sM \in\N$ and $$\mathcal A := \{\alpha_l\}_{l\in\N} = \N + (2s-\floor{2s})\N, \qquad \alpha_l < \alpha_{l+1} \quad \mbox{for all}\quad l\in\N,$$
    for two suitable sequences $\{a_l\}_{l\in\N}, \{\varphi_j\}_{j\in\N}$ independent of $M$. Moreover, if $s\in[1/2,1)$ then $\varphi_j$ can be chosen to vanish for all $j\neq 0$.
\end{Pro}

\begin{Rem}
    Observe that when $s=1/2$ we simply have $\{\alpha_l\}_{l\in\N} = \N$. Moreover, the powers of $\tau$ appearing in the definition of the phase function $\varphi$ emerge naturally from the proof of Proposition \ref{prop:approx-sol-rq}.
\end{Rem}

\begin{proof}
\textbf{Step 1.} \emph{(Stationary-nonstationary phase)} We let $\tilde a(x) := e^{i\tau\sum_{j=1}^{F_M} \tau^{-2sj}\varphi_j(x)} a(x)$, so that $u_0(x) = e^{i\tau\varphi_0(x)}\tilde a(x)$. Let $\chi \in C^{\infty}_c(\R^n)$ satisfy $0 \leq \chi \leq 1$, $\chi(\xi) = 1$ for $|\xi| \leq 1/2$, and $\chi(\xi) = 0$  for $|\xi| \geq 1$. We will use $\chi$ to give a precise meaning to the oscillatory integral below. We compute 
\begin{align*}
&(2\pi)^n (-\Delta)^s (e^{i\tau\varphi_0} \tilde a)(x) = \int e^{ix \cdot \xi} |\xi|^{2s} (e^{i\tau\varphi_0} \tilde a) \hat{\phantom{a}}(\xi) \,d\xi \\
 &= \lim_{\eps \to 0} \int e^{ix \cdot \xi} |\xi|^{2s} \chi(\eps \xi) (e^{i\tau\varphi_0} \tilde a) \hat{\phantom{a}}(\xi) \,d\xi  \\
 &= \lim_{\eps \to 0} \iint e^{i(x-y) \cdot \xi+ i\tau\varphi_0(y)} |\xi|^{2s} \chi(\eps \xi) \tilde a(y) \,dy \,d\xi \\
 &= \lim_{\eps \to 0} \tau^{2s+n} \iint e^{i\tau[(x-y) \cdot \xi+ \varphi_0(y)]} |\xi|^{2s} \chi(\eps \xi) \tilde a(y) \,dy \,d\xi.
\end{align*}

We wish to use the method of stationary phase to compute the asymptotics of this integral as $\tau \to \infty$. To deal with the minor problem that $|\xi|^{2s}$ is not smooth near $\xi=0$, write $|\xi|^{2s} = \chi(\tau^{\alpha} \xi) |\xi|^{2s} + (1-\chi(\tau^{\alpha} \xi)) |\xi|^{2s}$ where $0 < \alpha < 1$. Observe that $\int e^{i\tau(x-y) \cdot \xi} \chi(\tau^{\alpha} \xi) |\xi|^{2s} \,d\xi = \tau^{-(n+2s)\alpha}\eta(\tau^{1-\alpha}(x-y))$, where $\eta$ is smooth and all of its derivatives are bounded in $\R^n$. Thus 
\[
\iint e^{i\tau[(x-y) \cdot \xi+ \varphi_0(y)]} \chi(\tau^{\alpha} \xi) |\xi|^{2s}  \tilde a(y) \,dy \,d\xi = \tau^{-\alpha(n+2s)}\int e^{i\tau \varphi_0(y)} \eta(\tau^{1-\alpha}(x-y)) \tilde a(y) \,dy.
\]
Since $\nabla \varphi_0$ is nonvanishing near $\mathrm{supp(\tilde a)}$, nonstationary phase \cite[Theorem 7.7.1]{HO:analysis-of-pdos} implies that the last integral is $O(\tau^{-\alpha N})$ for any $N \geq 0$.

We now consider the expression 
\[
\lim_{\tau\to \infty} \tau^{2s+n} \iint e^{i\tau\Phi} (1-\chi(\tau^{\alpha} \xi)) |\xi|^{2s} \chi(\eps \xi) \tilde a(y) \,dy \,d\xi
\]
where $\Phi(y,\xi; x) := (x-y) \cdot \xi + \varphi_0(y)$. In order to deal with the part where $\xi$ is large, we note that there is $R > 0$ such that $|\nabla \varphi_0| \leq R/4$ in $\mathrm{supp}(\tilde a)$. We further write $|\xi|^{2s} = \chi(\xi/R) |\xi|^{2s} + (1-\chi(\xi/R)) |\xi|^{2s}$. Then in $\mathrm{supp}(1-\chi(\xi/R))$ we have $|\xi-\nabla \varphi_0(y)| \geq |\xi|/2 \geq R/4$ and also $|\nabla \Phi| \geq |\xi|/2 \geq R/4$. Applying nonstationary phase \cite[Theorem 7.7.1]{HO:analysis-of-pdos}, we have 
\[
\iint e^{i\tau\Phi} (1-\chi(\xi/R)) |\xi|^{2s} \chi(\eps \xi) \tilde a(y) \,dy \,d\xi = O(\tau^{-N})
\]
for any $N \geq 0$, uniformly over $0 < \eps < 1$.

It remains to consider the expression 
\[
\lim_{\tau\to \infty} \tau^{2s+n} \iint e^{i\tau\Phi} u(y,\xi) \,dy \,d\xi,
\]
where $u(y,\xi) = u_{\tau,\eps}(y,\xi) := \tilde a(y)(1-\chi(\tau^{\alpha} \xi))  |\xi|^{2s} \chi(\xi/R) $. Note that $u \in C^{\infty}_c(\R^{2n})$. We write $x_0$ instead of $x$ and rewrite the integral as 
\[
\iint e^{i\tau((x_0-y) \cdot \xi + \varphi_0(y))} u(y,\xi) \,dy \,d\xi = \iint e^{i\tau(\varphi_0(x+x_0) - x \cdot \xi)} u(x+x_0,\xi) \,dx \,d\xi.
\]
We now apply stationary phase \cite[Theorem 7.7.7]{HO:analysis-of-pdos} with $f(x,x_0) = \varphi_0(x+x_0)$ to obtain that 
\begin{align*}
 &\iint e^{i\tau(f(x,x_0) - x \cdot \xi)} u(x+x_0,\xi) \,dx \,d\xi \\
 &= e^{i\tau \varphi_0(x_0)} \left( \frac{2\pi}{\tau} \right)^n \sum_{\nu=0}^{N} \frac{1}{\nu!}\langle i D_x/\tau, D_{\xi} \rangle^{\nu} (e^{i\tau r_{x_0}(x)} u(x+x_0,\xi))|_{(x,\xi) = (0,\nabla \varphi_0(x_0))} + O(\tau^{-n-N/2})
 \\ & =: 
 e^{i\tau \varphi_0(x_0)} \left( \frac{2\pi}{\tau} \right)^n \sum_{\nu=0}^{N} \frac{S_\nu}{\nu!}+ O(\tau^{-n-N/2})
\end{align*}
where $r_{x_0}(x) := \varphi_0(x_0+x)-\varphi_0(x_0) - \nabla \varphi_0(x_0) \cdot x$.Observe that this implies $r_{x_0}(0) = \nabla r_{x_0}(0) = 0$. \\

\textbf{Step 2.} \emph{(High order terms)} In this step we clarify the dependency on $\tau$ of the higher order terms appearing in $(-\Delta)^su_0$. By letting
\begin{align*}
    S_\nu &  =: i^\nu \tau^{-\nu} \sum_{|\beta|=\nu}    D^\beta_\xi(|\xi|^{2s})|_{\xi = \nabla \varphi_0(y)}S_{\nu,\beta},
\end{align*}
we have that
\begin{align*}
    S_{\nu,\beta} = 
    e^{i\tau(\varphi-\varphi_0)(y)}\sum_{\substack{ \{ t_j \}_{j=0}^{F_M} \subset \N: \\ \sum_{j=0}^{F_M} t_j\leq |\beta|  }} \tau^{\sum_{j=0}^{F_M} t_{j}(1-2sj)} \sum_{\substack{\delta\leq\beta \\ |\delta|\leq |\beta|-\sum_{j=0}^{F_M} t_{j}}}  D^{\delta} a(y) \sum_{\substack{ \{ \eta_j \}_{j=0}^{F_M} \subset \N^n: \\ \sum_{j=0}^{F_M} \eta_j = \beta-\delta, \; |\eta_j|\geq t_j}}   R^{\beta,\delta, \mathcal N}_{\mathcal T}(y),
\end{align*}
where we let $\mathcal N:=\{\eta_j\}_{j=0}^{F_M}$, $\mathcal T:=\{t_j\}_{j=0}^{F_M}$, $\mathcal T':=\{t_j\}_{j=1}^{F_M}$ and
$$ R^{\beta,\delta, \mathcal N}_{\mathcal T}(y) := \binom{\beta}{\delta}\binom{\beta-\delta}{\beta-\delta-\eta_0} I_{\eta_0, t_0, r_y}(0) \prod_{j=1}^{F_M} \binom{\beta-\delta-\sum_{i=0}^{j-1}\eta_i}{\beta-\delta-\sum_{i=0}^j\eta_i }  I_{\eta_{j},t_{j},\varphi_{j}(\cdot+y)}(0).$$
In the above computation we used the fact that $r_y(0)=0$ and
\begin{align*}
    e^{-i\tau^\sigma v(x)}D^{\alpha}_x(e^{i\tau^{\sigma}v(x)}) =  \sum_{ j\leq |\alpha|} i^j\tau^{j\sigma} \sum_{\substack{ \rho_1, ... ,\rho_j \in \N_{>0}^n : \\ \rho_1+...+\rho_j=\alpha}}D^{\rho_1}_xv(x) \,...\, D^{\rho_j}_xv(x) =: \sum_{ j=0}^{|\alpha|} \tau^{j\sigma} I_{\alpha,j,v}(x).
\end{align*}
Thus 
\begin{align*}
    \sum_{\nu=0}^{N} \frac{S_\nu}{\nu!}  =  e^{i\tau(\varphi-\varphi_0)(y)}
    \sum_{\mu=0}^{N}  \sum_{\substack{ \{ t_j \}_{j=1}^{F_M} \subset \N: \\ \sum_{j=1}^{F_M} t_j\leq \mu  }} \tau^{-\mu+\sum_{j=1}^{F_M} t_{j}(1-2sj)} L_{\mu;\mathcal T'}a(y),
\end{align*}
where the differential operator $L_{\mu;\mathcal T'}$ of order  $\mu-\sum_{j=1}^{F_M} t_{j}$ is defined as
\begin{equation}
    \label{def-operators-L}
    L_{\mu;\mathcal T'}f(y) := \sum_{ |\delta|\leq \mu-\sum_{j=1}^{F_M} t_{j}} D^{\delta} f(y) \sum_{t_0=0}^{N-\mu} \frac{i^{\mu+t_0}}{({\mu+t_0})!}\sum_{\substack{\beta\geq\delta \\ |\beta|={\mu+t_0}}}  D^\beta_\xi(|\xi|^{2s})|_{\xi = \nabla \varphi_0(y)}    \sum_{\substack{ \{ \eta_j \}_{j=0}^{F_M} \subset \N^n: \\ \sum_{j=0}^{F_M} \eta_j = \beta-\delta, \\ |\eta_j|\geq t_j}}   R^{\beta,\delta, \mathcal N}_{\mathcal T}(y).
\end{equation}
Therefore we obtain  
\begin{align}\label{eq:before-eikonal}
    (-\Delta)^su_0(x)&=e^{i\tau\varphi(x)}\sum_{\mu=0}^{N}  \sum_{\substack{ \{ t_j \}_{j=1}^{F_M}\subset \N: \\ \sum_j t_j\leq \mu  }} \tau^{2s-\mu+\sum_{j=1}^{F_M} t_{j}(1-2sj)} L_{\mu;\mathcal T'}a(x) + O(\tau^{-n-N/2}).
\end{align}

\textbf{Step 3.} \emph{(The eikonal equation)} If $\sum_{j=1}^{F_M}t_j =\mu$ holds, then there exists a function $\lambda_{\mu;\mathcal T'}$ depending on the phase functions $\{\varphi_j\}_{j\in\N}$ such that $$L_{\mu;\mathcal T'}(y)f(y) = \lambda_{\mu;\mathcal T'}(y)f(y). $$  Using the fact that $I_{\alpha,|\alpha|,r_y}(0) =0$ for all $\alpha\neq 0$, and $I_{0,0,v}(0) =1$ for all $v$, we compute in particular the special cases
$$\lambda_{0;0,...}(y)=  |\nabla \varphi_0(y)|^{2s}$$
and, for all $j\geq 1$,
$$\lambda_{1;\mathcal T_j}(y)=-2s|\nabla \varphi_0(y)|^{2s-2}\nabla \varphi_0(y)   \cdot   \nabla \varphi_j(y),$$
where $\{\mathcal T_{j,i}\}_{i= 1}^{F_M}:=\delta_{i,j}$.
Therefore by assuming that the eikonal equation $|\nabla\varphi_0(x)|=r(x)$ holds we are left with
\begin{equation}\begin{split}\label{after-eikonal}
& e^{-i\tau\varphi(x)} ((-\Delta)^s + q(x) - \tau^{2s} r(x)^{2s}) u_0(x) =
\\ & \qquad  =  
\sum_{\mu=1}^{N}  \sum_{\substack{ \{ t_j \}_{j=1}^{F_M} \subset \N: \\ \sum_j t_j\leq \mu  }} \tau^{2s-\mu+\sum_{j=1}^{F_M} t_{j}(1-2sj)} L_{\mu;\mathcal T'}a(x)  + q(x) a(x) + O(\tau^{-n-N/2}).
\end{split}\end{equation}

\textbf{Step 4.} \emph{(Transport equations, $s\in [1/2,1)$)} Let us consider first $s\neq 1/2$. In this case it suffices to assume $\varphi_j\equiv 0$ for all $j\neq 0$, which immediately implies $L_{\mu;\mathcal T'}\equiv 0$ for all $\mathcal T'\neq 0$. In these assumptions equation \eqref{after-eikonal} becomes
\begin{equation}\begin{split}\label{final-decomposition-large-s}
& e^{-i\tau\varphi(x)} ((-\Delta)^s + q(x) - \tau^{2s} r(x)^{2s}) u_0(x) =
\\ & \qquad  =  
\sum_{l=0}^{A_M} \tau^{2s-1-\alpha_l} L_{1;0}a_l(x)+\sum_{\mu=2}^{N}\sum_{l=0}^{A_M} \tau^{2s-\mu-\alpha_l} L_{\mu;0}a_l(x) \\ & \qquad\quad + \sum_{l=0}^{A_M} \tau^{-\alpha_l}q(x) a_l(x) + O(\tau^{-n-N/2}),
\end{split}\end{equation}
where we used the assumption on the amplitude $a$. Observe that whenever $l,l'\in\N$ are such that \begin{equation}\label{absorption-condition-large-s}
    2s-1-\alpha_l = -\alpha_{l'} \qquad\mbox{ or }\qquad 2s-1-\alpha_l=2s-\mu-\alpha_{l'},
\end{equation} 
 then it must be $l'<l$, since $0< 2s-1$ and $2s-\mu<2s-1$ for all $\mu\in\{2,..,N\}$. In order to ensure that for each $l'\in\N$ there exists a $l\in\N$ verifying either of the conditions in \eqref{absorption-condition-large-s}, we need to verify that
 $$ 2s-1+\mathcal A\subseteq \mathcal A \qquad\mbox{ and }\qquad \mu-1+ \mathcal A\subseteq \mathcal A$$
 for all $\mu\in\{2,...,N\}$, which is indeed the case for $\{a_l\}_{l\in\N}=\N+(2s-1)\N$ as in the statement. Therefore, all the summands in the second and third term on the right-hand side of \eqref{final-decomposition-large-s} (up to indices $A'_M, A''_M$, respectively) can be absorbed by the ones in the first term, because the corresponding powers of $\tau$ coincide. 
By assuming that the coefficients associated to each power of $\tau$ vanish, we obtain a sequence of differential equations in the unknowns $\{a_l\}_{l=0}^{A_M}$. Because of our previous discussion, these are all transport equations for the amplitude functions $a_l$ whose source terms depend only on $\{a_j\}_{j<l}$. We are therefore able to obtain the amplitudes recursively, which leaves us with
\begin{equation}\begin{split}\label{final-s-large}
    e^{-i\tau\varphi(x)} & ((-\Delta)^s + q(x) - \tau^{2s} r(x)^{2s}) u_0(x) = \\ & = 
\sum_{l=A'_M}^{A_M} \sum_{\mu=2}^{N} \tau^{2s-\mu-\alpha_l} L_{\mu;0}a_l(x) + \sum_{l=A''_M}^{A_M} \tau^{-\alpha_l}q(x) a_l(x) + O(\tau^{-n-N/2}) . 
\end{split}\end{equation} 
\vspace{2mm}

Consider now $s=1/2$. In this case the right-hand side of equation \eqref{final-decomposition-large-s} becomes 
$$ \sum_{l=0}^\infty \tau^{-\alpha_l} (L_{1;0}+q(x))a_l(x)+\sum_{\mu=2}^{N}\sum_{l=0}^\infty \tau^{1-\mu-\alpha_l} L_{\mu;0}a_l(x) + O(\tau^{-n-N/2}), $$
and thus we can repeat the same procedure as before using the new first order operator $L_{1;0}+q(x)$ in place of $L_{1;0}$. 
\vspace{2mm}

\textbf{Step 5.} \emph{(Transport equations, $s\in (0,1/2)$)} In this case we write
\begin{equation*}\begin{split}
& e^{-i\tau\varphi(x)} ((-\Delta)^s + q(x) - \tau^{2s} r(x)^{2s}) u_0(x) =
\\ & \qquad  =  
\tau^{2s-1} L_{1;0}a(x) + \sum_{\mu=2}^{N}  \sum_{\substack{ \{ t_j \}_{j=1}^{F_M} \subset \N: \\ \sum_j t_j\leq \mu  }} \tau^{2s-\mu+\sum_j t_{j}(1-2sj)} L_{\mu;\mathcal T'}a(x)
\\ & \qquad\quad +  \sum_{j=2}^{F_M} \tau^{2s(1-j)} L_{1;\mathcal T_j}a(x) +  L_{1;\mathcal T_1}a(x)  + q(x) a(x) + O(\tau^{-n-N/2}).
\end{split}\end{equation*} By solving the first transport equation
\begin{equation}\label{eq:transport-phase} q(x)- 2s|\nabla \varphi_0(x)|^{2s-2}\nabla \varphi_0(x)   \cdot   \nabla \varphi_1(x)= q(x) + \lambda_{1;\mathcal T_1}(x)=0 \end{equation}
for the phase function $\varphi_1$, we are left with
\begin{equation*}\begin{split}
& e^{-i\tau\varphi(x)} ((-\Delta)^s + q(x) - \tau^{2s} r(x)^{2s}) u_0(x) =
\\ & \qquad  =   
\sum_{\mu=2}^{N}  \sum_{\substack{ \{ t_j \}_{j=1}^{F_M} \subset \N: \\ \sum_j t_j< \mu  }} \sum_{l=0}^{A_M} \tau^{2s(1-\sum_j jt_{j})-(\mu-\sum_j t_{j})-\alpha_l} L_{\mu;\mathcal T'}a_l(x)
\\ & \qquad\quad + \left(\sum_{\mu=2}^{N}  \sum_{\substack{ \{ t_j \}_{j=1}^{F_M} \subset \N: \\ \sum_j t_j= \mu  }} \tau^{2s(1-\sum_j jt_{j})} \lambda_{\mu;\mathcal T'}(x) + \sum_{j=2}^{F_M} \tau^{2s(1-j)} \lambda_{1;\mathcal T_j}(x)\right)a(x)
\\ & \qquad\quad + \sum_{l=0}^{A_M} \tau^{2s-1-\alpha_l} L_{1;0}a_l(x)  + O(\tau^{-n-N/2}),
\end{split}\end{equation*}
where all the powers of $\tau$ are negative. The term in parentheses can be rewritten as
\begin{align*}
    S_1+ S_2 :&= \sum_{k=2}^{F_M} \tau^{2s(1-k)} \left(\lambda_{1;\mathcal T_k}(x)+ \sum_{\mu=2}^{N}  \sum_{\substack{ \{ t_j \}_{j=1}^{F_M} \subset \N: \\ \sum_j t_j= \mu, \; \sum_j jt_j = k  }} \lambda_{\mu;\mathcal T'}(x)\right) \\ & \quad + \sum_{\mu=2}^{N}  \sum_{\substack{ \{ t_j \}_{j=1}^{F_M} \subset \N: \\ \sum_j t_j= \mu, \; \sum_j jt_j > F_M  }} \tau^{2s(1-\sum_j jt_{j})}  \lambda_{\mu;\mathcal T'}(x),
\end{align*}
and we see that $S_1$ vanishes as soon as
\begin{equation}\label{recursion-phase}
    2s|\nabla \varphi_0(x)|^{2s-2}\nabla \varphi_0(x)   \cdot   \nabla \varphi_k(x) = \sum_{\mu=2}^{N}  \sum_{\substack{ \{ t_j \}_{j=1}^{F_M} \subset \N: \\ \sum_j t_j= \mu, \; \sum_j jt_j = k  }} \lambda_{\mu;\mathcal T'}(x)
\end{equation}
holds for all $k\in\{2,..., F_M\}$. This is a sequence of transport equations for the phase functions $\varphi_k$, where the right-hand sides only depend on the functions $\{\varphi_j\}_{j<k}$. This follows from the fact that, since $\mu\geq 2$, necessarily the largest $j$ for which $t_j$ does not vanish must be strictly smaller than $k$. Having already determined both $\varphi_0$ and $\varphi_1$, equation \eqref{recursion-phase} allows us to compute all the remaining phase functions recursively. Eventually, we are left with
\begin{equation}\begin{split}\label{final-decomposition}
& e^{-ik\varphi(x)} ((-\Delta)^s + q(x) - \tau^{2s} r(x)^{2s}) u_0(x) =
\\ & \qquad  =   
\sum_{\mu=2}^{N}  \sum_{\substack{ \{ t_j \}_{j=1}^{F_M} \subset \N: \\ \sum_j t_j< \mu  }} \sum_{l=0}^{A_M} \tau^{2s(1-\sum_j t_{j}j)-(\mu-\sum_j t_{j})-\alpha_l} L_{\mu;\mathcal T'}a_l(x)
\\ & \qquad\quad + \sum_{l=0}^{A_M} \tau^{2s-1-\alpha_l} L_{1;0}a_l(x)  + S_2 a+ O(\tau^{-n-N/2}).
\end{split}\end{equation}
Observe that we have $2s(1-\sum_{j=1}^{F_M} t_{j}j)-(\mu-\sum_{j=1}^{F_M} t_{j})< 2s-1$ for all choices of $\mu, \mathcal T'$ as in the first term on the right-hand side. This implies that if $l,l'\in\N$ are such that
\begin{equation}\label{absorption-condition}
    2s(1-\sum_{j=1}^{F_M} t_{j}j)-(\mu-\sum_{j=1}^{F_M} t_{j})-\alpha_l = 2s-1-\alpha_{l'},
\end{equation} 
then necessarily $l<l'$. The fact that for all $l\in \N$ there exists a $l'\in\N$ such that \eqref{absorption-condition} holds follows from $$2s\sum_{j=1}^{F_M} t_{j}j+(\mu-\sum_{j=1}^{F_M} t_{j}-1)  +\mathcal A \subseteq \mathcal A,$$
which holds for $\{\alpha_j: j\in\N\}=\N+2s\N$. Therefore, all the summands in the first term on the right-hand side of \eqref{final-decomposition} (up to a certain index $A'_M$) can be absorbed by the ones of the second term, and by comparing the powers of $\tau$ we obtain a sequence of transport equations for the amplitude functions $a_l$, where the source terms depend only on $\{a_j\}_{j<l}$. This allows us to compute all the amplitude functions recursively, which leaves us with
\begin{equation}\begin{split}\label{final-s-small}
& e^{-i\tau\varphi(x)} ((-\Delta)^s + q(x) - \tau^{2s} r(x)^{2s}) u_0(x) =
\\ & \qquad  =   
\sum_{\mu=2}^{N}  \sum_{\substack{ \{ t_j \}_{j=1}^{F_M} \subset \N: \\ \sum_j t_j< \mu  }} \sum_{l=A'_M}^{A_M} \tau^{2s(1-\sum_j t_{j}j)-(\mu-\sum_j t_{j})-\alpha_l} L_{\mu;\mathcal T'}a_l(x)
\\ & \qquad\quad  + \sum_{\mu=2}^{N}  \sum_{\substack{ \{ t_j \}_{j=1}^{F_M} \subset \N: \\ \sum_j t_j= \mu, \; \sum_j jt_j > F_M  }} \tau^{2s(1-\sum_j jt_{j})}  \lambda_{\mu;\mathcal T'}(x)a(x) + O(\tau^{-n-N/2}).
\end{split}\end{equation}

\textbf{Step 6.} \emph{(Final estimates)} We will now use equations \eqref{final-s-large} and \eqref{final-s-small} to show that there exists $C_{M,\beta} > 0$ such that
$$\|((-\Delta)^s-\tau^{2s}r^{2s}+q)u_0\|_{H^\beta_{scl}(\Omega)} \leq C_{M,\beta} \tau^{-M}.$$
If $s\geq 1/2$, then 
\begin{align*}
    \|((-\Delta)^s&-\tau^{2s}r^{2s}+q)u_0\|_{H^\beta_{scl}(\Omega)} \leq \\ &  \leq c_N \tau^{-n-N/2} +  \tau^{2s-2-\alpha_{A'_M}}\| L_{2;0}a_{A'_M} \|_{H^\beta_{scl}(\Omega)} + \tau^{-\alpha_{A''_M}}\| qa_{A''_M} \|_{H^\beta_{scl}(\Omega)} .
\end{align*}
The wanted result is obtained by choosing $N= 2M$ and $A'_M, A''_M$ so large that $M\leq \min\{\alpha_{A'_M}, \alpha_{A''_M} \}$, which is possible by choosing $A_M$ large enough. Similarly, if $s<1/2$ we have
\begin{align*}
    \|((-\Delta)^s&-\tau^{2s}r^{2s}+q)u_0\|_{H^\beta_{scl}(\Omega)} \leq \\ & \leq  c_N \tau^{-n-N/2} +  \tau^{2s(1-F_M)}\| L_{\mu;\mathcal T'}a \|_{H^\beta_{scl}(\Omega)} + \tau^{2s-\alpha_{A'_M}}\| L_{\mu;\mathcal T'}a_{A'_M} \|_{H^\beta_{scl}(\Omega)} ,
\end{align*}
which implies the result for $N= 2M$ and $A_M,F_M$ large enough. Using the relation
$$ \alpha_{K_sh}\geq h \quad \mbox{for all } h\in\N, \qquad K_s:=\begin{cases}
    (2s-\floor{2s})^{-1}, & s\in (0,\frac12)\cup(\frac12,1) \\ 1, & s=\frac12
\end{cases}, $$
it is easy to see that choosing $A_M, F_M = \ceil{3K_s}M$ suffices in both regimes. Finally, the constant $C_{M,\beta}$ is obtained by observing that all the $L_{\mu;\mathcal T'}$ are local differential operators, and all the amplitudes $a_l$ belong to $C^\infty(\overline\Omega)$ by construction, which implies that all the quantities of the kind $\| L_{\mu;\mathcal T'}a_{l} \|_{H^\beta_{scl}(\Omega)}$ are bounded.
\end{proof}

\begin{Rem}\label{rem:why-small-s-hard}
    Observe that the geometrical optics solutions constructed in Proposition \ref{prop:approx-sol-rq} for the two regimes $s\in (0,\frac 12)$ and $s\in[\frac 12,1)$ differ substantially. While in the former case there are many phase function associated to both positive and negative powers of $\tau$, in the latter case we have a unique phase function associated to $\tau$. This difference is due to the computations in Step 4: if we were to assume a unique phase function also in the regime $s\in (0,\frac 12)$, which would lead to equation \eqref{final-decomposition-large-s}, then because $2s-1<0$ we would obtain $l<l'$ when equation (\ref{absorption-condition-large-s}.1) holds. This in turn would imply that in the transport equation for $a_l$ there would appear terms depending on $a_j$ for $j>l$, which would impede the recursion argument. Thus we have shown that it is impossible to get geometrical optics with a single phase function in the regime $s\in(0,\frac 12)$. This fact will have important consequences in Section \ref{sec:applications}, where we study applications to inverse problems.
\end{Rem}

We complement the results of Proposition \ref{prop:approx-sol-rq} with the following Corollary:

\begin{Cor}
    Making use of the fractional Liouville reduction formula
$$ \mathbf C^s_\gamma u = \gamma^{1/2}\left((-\Delta)^s(\gamma^{1/2}u) - u (-\Delta)^s\gamma^{1/2}\right) $$
(see \cite{C20}), we can compute
\begin{align*}
    \mathbf C^s_\gamma u - \tau^{2s}u =
    \gamma^{1/2}\left((-\Delta)^s - \tau^{2s} \gamma^{-1} -\gamma^{-1/2}(-\Delta)^s\gamma^{1/2} \right)(\gamma^{1/2}u) ,
\end{align*}
which we recognize to be an operator of the kind considered in the above Proposition, with $r(x):= \gamma^{-1/2s}(x)$ and $q(x):= -\gamma^{-1/2}(x)(-\Delta)^s\gamma^{1/2}(x)$. In particular, we deduce that it is possible to find an approximate solution of the form $u(x)= e^{i\tau\varphi(x)}\gamma^{-1/2}(x)a(x)$, with $\varphi(x)$ and $a(x)$ determined by the eikonal and transport equations as above.
\end{Cor}

\section{Upgrading approximate solutions to exact solutions}\label{sec:exact}

In this section we show how to construct exact solutions to the problem $Lu=0$ in $\Omega$ starting from approximate geometrical optics solutions $u_M$ verifying the property
$$\| Lu_M \|_{L^2(\Omega)}\leq C_M\tau^{-M},$$
where $L:=(-\Delta)^{s}-\tau^{2s}r^{2s}+q$. In order to produce an exact solution $u = u_M + v_M$, it is enough to solve the equation $Lv_M = -Lu_M$ with suitable estimates. To state the solvability result, we will consider the small parameter $h := \tau^{-1}$. The semiclassical Sobolev norm is defined by 
\[
\norm{u}_{H^{\alpha}_{scl}(\R^n)} := \norm{\langle hD \rangle^{\alpha} u}_{L^2(\R^n)},
\]
where $\langle \xi \rangle = (1+|\xi|^2)^{1/2}$ and $D = -i\nabla$. The $H^{\alpha}_{scl}(\Omega)$ norm is defined by restriction.

The solvability results are as follows.  In order to have good high frequency estimates for all large $\tau$ in the case of the operator $(-\Delta)^{s}-\tau^{2s}r^{2s}+q$, we need to assume that the refraction index $r(x)$ is \emph{nontrapping} in $\overline{\Omega}$, i.e.\ that any geodesic of the Riemannian metric $g_{jk}(x) = r(x)^{2} \delta_{jk}$ reaches $\partial \Omega$ in finite time in both directions.

\begin{Pro}\label{first-solvability-proposition}
Let $\Omega \subset \R^n$ be open and bounded, $\alpha\in\R$ and $ s\in(0, 1)$. Let $L:=(-\Delta)^{s}-\tau^{2s}r^{2s}+q$, with $r,q \in C^{\infty}(\overline{\Omega})$, $r$ positive and nontrapping. There are $C, h_0 > 0$ such that for any $f \in H^{\alpha}(\Omega)$ the equation 
\[
Lu = f \text{ in $\Omega$}
\]
has a solution $u \in H^{2s-1+\alpha}(\Omega)$ satisfying 
\[
\norm{u}_{H^{2s-1+\alpha}_{scl}(\Omega)} \leq C h^{2s-1} \norm{f}_{H^{\alpha}_{scl}(\Omega)}
\]
whenever $0 < h < h_0$.
\end{Pro}

The above result will follow by duality from the next a priori estimate.

\begin{Pro} \label{prop_apriori}
Let $\Omega \subset \R^n$ be open and bounded, $\alpha\in\R$ and $ s\in(0, 1)$. Let $L:=(-\Delta)^{s}-\tau^{2s}r^{2s}+q$, with $r,q \in C^{\infty}(\overline{\Omega})$, $r$ positive and nontrapping. There are $C, h_0 > 0$ such that the estimate 
\[
\norm{u}_{H^{2s-1+\alpha}_{scl}} \leq C h^{-1} \norm{h^{2s}Lu}_{H^{\alpha}_{scl}}
\]
holds whenever $u \in C^{\infty}_c(\Omega)$ and $0 < h < h_0$.
\end{Pro}

The estimate in Proposition \ref{prop_apriori} is a high frequency resolvent type estimate, but it is stated in a bounded domain and therefore one does not need to pay attention to behavior at infinity. We will prove the estimate by a positive commutator argument as in \cite{MSS22}, which considered the case $s=1$. The main difference in our case is that the symbol of the fractional Laplacian is not smooth at $\xi=0$, so one cannot use $C^{\infty}$ pseudodifferential calculus. We will prove this case by a direct argument based on freezing coefficients. \\

We let $P:= (-h^2 \Delta)^s - r^{2s} + h^{2s}q$ with semiclassical principal symbol $$p(x,\xi) := |\xi|^{2s}-r(x)^{2s},$$ and we define $\tilde P$ as the operator generated by $p(x,\xi)$. Observe that $\|(P-\tilde P)u\|\lesssim h^{2s}\|u\|$ for all $u\in C^\infty_c(\Omega)$.  Choose $\chi \in C^{\infty}_c(\R^n)$ with $\chi = 1$ when $|\xi| \leq \frac{1}{4} c_0$ and $\chi = 0$ when $|\xi| \geq \frac{1}{2} c_0$, where the constant $c_0>0$ is such that $r(x)\geq c_0$ for all $x\in\mathbb R^n$. We write 
\[
B_0 u := \chi(h D) u, \qquad Bu := (I - B_0)u.
\]
Note that $B_0 u$ is supported in phase space away from $p^{-1}(0)$. \\

The proof of Proposition \ref{prop_apriori} is based on the following two lemmas, which consider the case $\alpha=1$. Below we will write $A \lesssim B$ if $A \leq CB$ where $C$ is independent of $h \in (0,h_0)$ and of $u \in C^{\infty}_c(\Omega)$. The first lemma is an estimate in the elliptic region, and comes with a power of $h$ that is better than in the main estimate.

\begin{Lem} \label{lemma_bzerou}
For any $\eps > 0$ there is $h_0 > 0$ such that 
\[
\norm{B_0 u}_{H^{2s-1}_{scl}} \lesssim \norm{Pu}_{L^2} + \eps \norm{u}_{H^{2s-1}_{scl}}
\]
when $0 < h < h_0$.
\end{Lem}

The next lemma requires the assumption that $r$ is nontrapping, and uses a positive commutator argument to give an estimate away from $\xi=0$.

\begin{Lem} \label{lemma_bu}
For any $\eps > 0$ there are $h_0, C_{\eps} > 0$ such that 
\[
\norm{B u}_{H^{2s-1}_{scl}} \lesssim \eps \norm{u}_{H^{2s-1}_{scl}} + C_{\eps} h^{-1} \norm{Pu}_{L^2} + \norm{B_0 u}_{H^{2s-1}_{scl}}.
\]
when $0 < h < h_0$. Here the implied constant is independent of $\eps$.
\end{Lem}

Assuming these lemmas, the proof of Proposition \ref{prop_apriori} is given as follows.

\begin{proof}[Proof of Proposition \ref{prop_apriori}]
Let first $\alpha=0$. Adding the estimates in Lemmas \ref{lemma_bzerou} and \ref{lemma_bu} yields 
\begin{align*}
\norm{u}_{H^{2s-1}_{scl}} &\leq \norm{B_0 u}_{H^{2s-1}_{scl}} + \norm{Bu}_{H^{2s-1}_{scl}} \\
 &\lesssim \norm{Pu}_{L^2} + \eps \norm{u}_{H^{2s-1}_{scl}} + C_{\eps} h^{-1} \norm{Pu}_{L^2} + \norm{B_0 u}_{H^{2s-1}_{scl}} \\
 &\lesssim \eps \norm{u}_{H^{2s-1}_{scl}} + C_{\eps} h^{-1} \norm{Pu}_{L^2}.
\end{align*}
Choosing $\eps$ small enough allows us to absorb the first term on the right, which proves the estimate when $\alpha=0$. The same estimate is also true in a slightly larger domain $\Omega_1$. We apply the estimate in $\Omega_1$ to $\psi (1-h^2\Delta)^{\alpha/2} u$ where $\psi \in C^{\infty}_c(\Omega_1)$ is a cutoff function with $\psi = 1$ near $\overline{\Omega}$. Using pseudolocal estimates for the commutator $[P,\psi (1-h^2\Delta)^{\alpha/2}]$ and the fact that $\mathrm{supp}(u) \subset \Omega$ gives the estimate for general $\alpha$.
\end{proof}

Next, we prove Lemmas \ref{lemma_bzerou} and \ref{lemma_bu}.

\begin{proof}[Proof of Lemma \ref{lemma_bzerou}]
We first let $x_0 \in \overline{\Omega}$ and consider the frozen coefficient operator $P_{x_0}$ of semiclassical symbol $p(x_0,\xi)$. Since $r(x) \geq c_0$ and $\mathrm{supp}(\chi) \subset \{ \abs{\xi} \leq c_0/2 \}$, we see that
\begin{align*}
    \|B_0 u\|_{H^{2s-1}_{scl}} & = \|\langle h\xi \rangle^{2s-1} \chi(h\xi)\hat u\|_{L^2} = \left\|\frac{\langle h\xi \rangle^{2s-1} \chi(h\xi)}{p(x_0,h\xi)}p(x_0,h\xi)\hat u\right\|_{L^2}
    \\ & \leq
    \left\| \frac{\langle h\xi \rangle^{2s-1} \chi(h\xi)}{p(x_0,h\xi)} \right\|_{L^\infty}\|p(x_0,h\xi)\hat u\|_{L^2}
    \\ & \lesssim
   \| P_{x_0} u\|_{L^2}.
\end{align*}
The implied constants are independent of $x_0\in\overline\Omega$.

In order to obtain an argument for variable coefficients, first fix $\delta>0$, and then cover $\overline\Omega$ with $N_\delta$ balls $U_j:=B(x_j,\rho)$ in such way that $|r(x)^{2s}-r(x_j)^{2s}|\leq \delta$ for all $x\in\overline{B}_j$. Let now $\{\chi_j\}_{j=1}^N$ with $0 \leq \chi_j \leq 1$ be a partition of unity subordinate to the cover $\{U_j\}$ such that $\sum \chi_j^2 = 1$ near $\overline{\Omega}$. We can also arrange so that $\norm{\partial^{\gamma} \chi_j}_{L^{\infty}} \leq C_{\delta,M}$ for $\abs{\gamma} \leq M$, uniformly over $j$. Then for any fixed $j$, one has  

\begin{align}
    \|B_0(\chi_ju)\|_{H^{2s-1}_{scl}} &\lesssim  \|P_{x_j}(\chi_j u)\|_{L^2}
    \notag \\ & \leq
     \|\tilde P(\chi_j u)\|_{L^2} + \| (r(x_j)^{2s}-r(x)^{2s} )(\chi_ju)\|_{L^2} 
    \notag \\
     & \leq \|\chi_j \tilde P u\|_{L^2} + \| [\tilde P, \chi_j] u \|_{L^2} + \delta \|\chi_ju\|_{L^2}. \label{bzero_first}
\end{align}
On the left hand side, we also estimate 
\[
 \|B_0(\chi_ju)\|_{H^{2s-1}_{scl}} \geq  \|\chi_j B_0 u\|_{H^{2s-1}_{scl}} -  \|[\chi_j, B_0]u\|_{H^{2s-1}_{scl}}.
\]
The last commutator satisfies the semiclassical estimate 
\[
\|[\chi_j, B_0]u\|_{H^{2s-1}_{scl}} \leq C_{\delta} h \|u\|_{H^{2s-1}_{scl}}.
\]

For the other commutator term, we use the fact that $\chi(D)$ is smoothing and $(1-\chi(D))\tilde P$ is a classical pseudodifferential operator to estimate
\[
\|[\chi_j,\tilde P]u\|_{L^2} \leq \|[\chi_j,\chi(D)\tilde P]u\|_{L^2} + \|[\chi_j,(1-\chi(D))\tilde P]u\|_{L^2} \leq C_{\delta}h \|u\|_{H^{2s-1}}.
\]
Thus, after squaring the estimate \eqref{bzero_first} we get 
\[
\|\chi_j B_0 u \|_{H^{2s-1}_{scl}}^2 \lesssim  \|\chi_j\tilde Pu\|_{L^2}^2 + C_{\delta} h^2 \|u\|_{H^{2s-1}_{scl}}^2 + \delta^2 \|\chi_j u\|_{L^2}^2.
\]

Adding these estimates and writing $\psi = \sum \chi_j^2$ gives 
\begin{equation} \label{psibzerou}
\|\psi B_0 u\|_{H^{2s-1}_{scl}}^2 \lesssim  \|\tilde Pu\|_{L^2}^2 + C_{\delta} h^2 \|u\|_{H^{2s-1}_{scl}}^2 + \delta^2 \| u\|_{L^2}^2.
\end{equation}
Since $\psi = 1$ near $\overline{\Omega}$ and $u \in C^{\infty}_c(\Omega)$, pseudolocal estimates give $\norm{(1-\psi)B_0 u}_{H^{2s-1}_{scl}} \lesssim h \|u\|_{H^{2s-1}_{scl}}$. Thus taking square roots in \eqref{psibzerou} gives 
\begin{equation}\label{almost-there-lemma-1}
   \|B_0 u\|_{H^{2s-1}_{scl}} \lesssim  \|Pu\|_{L^2} + C_{\delta} h \|u\|_{H^{2s-1}_{scl}} + (\delta + h^{2s} )\| u\|_{L^2}. 
\end{equation}
Here we have used the fact that $\|(P-\tilde P)u\|_{L^2}\lesssim h^{2s} \| u\|_{L^2}$.

Given $\eps > 0$, we first choose $\delta$ and $h_0$ so that $\delta + h^{2s} \leq \eps/2$ for $h < h_0$. Then we choose $h_0$ possibly even smaller so that $C_{\delta} h \leq \eps/2$ for $h < h_0$. After replacing $\eps$ by $\eps/C$ where $C$ is the implied constant, we obtain the required result.
\end{proof}

\begin{proof}[Proof of Lemma \ref{lemma_bu}]
Define the operator $P':= r(x)^{-2s}P$. Since $r \geq c_0$ in $\R^n$, it is enough to prove the estimate for $P'$ instead than $P$. We shall use the \emph{positive commutator method}, which requires the construction of a self-adjoint semiclassical pseudodifferential operator $A$ in $\R^n$ of order $2s-1$ (called a \emph{commutant}) such that the following properties hold for all $u\in C^\infty_c(\Omega)$:
\begin{enumerate}
    \item $\|Au\|_{L^2}\lesssim \|u\|_{H^{2s-1}_{scl}}$,
    \item $(i[P',A]u,u)_{L^2}\geq ch\|Bu\|^2_{H^{2s-1}_{scl}} - Ch\|B_0 u\|^2_{H^{2s-1}_{scl}}$.
\end{enumerate}

If such operator $A$ can be constructed, then we have
\begin{align*}
ch\|Bu\|^2_{H^{2s-1}_{scl}} - Ch\|B_0 u\|^2_{H^{2s-1}_{scl}} & \leq (i[P',A]u,u)_{L^2}
\\ & \leq
2 \, \|Au\|_{L^2} \, \|P'u\|_{L^2}
\\ & \leq
\varepsilon^2 h \|Au\|_{L^2}^2 + (\varepsilon^2 h)^{-1}\|P'u\|_{L^2}^2
\\ & \leq
c_1\varepsilon^2 h \|u\|^2_{H^{2s-1}_{scl}} + (\varepsilon^2 h)^{-1}\|P'u\|_{L^2}^2,
\end{align*}
and therefore
$$ \|Bu\|^2_{H^{2s-1}_{scl}}  \lesssim \varepsilon^2 \|u\|^2_{H^{2s-1}_{scl}} + \|B_0 u\|^2_{H^{2s-1}_{scl}} + \varepsilon^{-2}h^{-2}\|P'u\|^2_{L^2}. $$
This proves the required estimate.

In order to prove the existence of a suitable commutant $A$, we follow the same method as in the proof of \cite[Lemma 3.2]{MSS22}. Consider the simple manifold  $(\Omega,g)$, where the metric $g$ is given by $$ g|_{(x_1, ..., x_n)} := r(x)^{-2}\delta_{ij} dx^i\otimes dx^j, \qquad |g|:= \det g_{ij} = r(x)^{-2n}, $$
$$ \langle X,Y\rangle_g := r(x)^{-2} X\cdot Y, \qquad |X|_g := r(x)^{-1} |X|. $$ Let $$p'(x,\xi):=r(x)^{-2s}h^{2s}|\xi|^{2s} = h^{2s}|\xi|_g^{2s}$$ be the semiclassical principal symbol of $P'$, let $H_{p'}$ be its Hamiltonian vector field, and let $a\in C^\infty(S^*\Omega)$ be a function such that $H_{p'}a>0$ in $S^*\Omega$. Observe that such function exists, because $(\Omega,g)$ is nontrapping. If $\Omega\Subset \Omega_2$, we can extend $a$ smoothly to $T^*\Omega_2$ as a symbol that is homogeneous of degree $2s-1$ for $|\xi|_g\geq 1$, and such that
$$H_{p'}a(x,\xi)\gtrsim |\xi|_g^{2s}, \qquad \xi\in T^*\Omega_1,\; |\xi|_g\sim 1, $$
for some $\Omega_1$ such that $\Omega\Subset \Omega_1\Subset \Omega_2$.

The semiclassical differential operator $A$ of order $2s-1$ is now obtained by Weyl quantization of its real valued symbol $a$. Because $p'$ is not smooth, we are not allowed to use the semiclassical estimate for the commutator $[P',A]$. However, recall that by definition $B_0 := \chi(h D) $ and $B:=I-B_0$, with 
$$ \chi(\xi)=1 \quad \mbox{ if } \, |\xi| \leq \frac{c_0}{4}, \qquad  \chi(\xi)=0 \quad \mbox{ if } \, |\xi| \geq \frac{c_0}{2}. $$
Let us similarly define the operators $\tilde B_0 := \tilde\chi(h D) $ and $\tilde B:=I-\tilde B_0$, with 
$$ \tilde\chi(\xi)=1 \quad \mbox{ if } \, |\xi| \leq \frac{c_0}{16}, \qquad  \tilde\chi(\xi)=0 \quad \mbox{ if } \, |\xi| \geq \frac{c_0}{8}. $$ 
Then we can compute
\begin{align*}
    [P',A]B & = [\tilde B P',A]B  + [\tilde B_0 P',A] B 
    \\ & = [\tilde B P',A]B + \tilde B_0 P'A B - A\tilde B_0 P' B  
    \\ & = [\tilde B P',A]B + \tilde B_0 P' [A, B] + \tilde B_0 P' B A - A[\tilde B_0, P'B]
    \\ & = [\tilde B P',A]B + \tilde B_0 P' [A, B] + [\tilde B_0, P'B] A - A[\tilde B_0, P'B].
\end{align*}
Here we used that $\tilde B_0 P' B= [\tilde B_0, P'B]$, which holds because $B\tilde B_0=0$. Observe that the operators appearing in the commutators on the right-hand side have smooth symbols, and so we can use the semiclassical estimate to deduce
$$ \| [P',A]Bu \|_{H^{1-2s}_{scl}} \lesssim h \| Bu \|_{H^{2s-1}_{scl}}. $$
Moreover, using a freezing coefficients argument followed by partition of unity as in the proof of Lemma \ref{lemma_bzerou}, we are able to deduce that
\begin{align*}
    \|\chi_j P'AB_0u\| & \leq \|P'A\chi_j B_0u\| + \|[\chi_j ,P'A]B_0u\|
    \\ & \lesssim
    \|P'_{x_0}A_{x_0}\chi_j B_0u\| + \|\left( r(x)^{2s}a(x,\xi) - r(x_0)^{2s}a(x_0,\xi) \right)\chi_j B_0u\| + C_\delta\|B_0u\|
    \\ & \lesssim
    (1 + \delta)\|\chi_j B_0u\| + C_\delta\|B_0u\|,
\end{align*}
which reveals that $ \|P'AB_0u\| \lesssim h \|B_0u\| $. Because a similar estimate holds for $AP'B_0u$ as well, we deduce that
$$ \| [P',A]B_0u \|_{H^{1-2s}_{scl}} \lesssim h\| B_0u \|_{H^{2s-1}_{scl}}. $$

It now follows from the semiclassical Gårding inequality and the Cauchy-Schwarz inequality with $\varepsilon$ that
\begin{align*}
    (i[P',A]u,u)_{L^2} & = (i[P',A]Bu,Bu)_{L^2} + (i[P',A]Bu,B_0u)_{L^2} + (i[P',A]B_0u,u)_{L^2} 
    \\ & \geq
    c'h\|Bu\|^2_{H^{2s-1}_{scl}} -  \|i[P',A]Bu\|_{H^{1-2s}_{scl}}\|B_0u\|_{H^{2s-1}_{scl}}
    \\ & \quad - \|i[P',A]B_0u\|_{H^{1-2s}_{scl}}\left(\|Bu\|_{H^{2s-1}_{scl}} + \|B_0u\|_{H^{2s-1}_{scl}}\right)
    \\ & \geq
    c'h\|Bu\|^2_{H^{2s-1}_{scl}} - 2h \|Bu\|_{H^{2s-1}_{scl}}\|B_0u\|_{H^{2s-1}_{scl}} - h\|B_0u\|^2_{H^{2s-1}_{scl}}
     \\ & \geq
    (c'-\varepsilon) h\|Bu\|^2_{H^{2s-1}_{scl}} - C_\varepsilon h\|B_0u\|^2_{H^{2s-1}_{scl}}.
\end{align*}
This proves the second commutator inequality.
\end{proof}

 Finally, we prove Proposition \ref{first-solvability-proposition} in light of Proposition \ref{prop_apriori}:
\begin{proof}[Proof of Proposition \ref{first-solvability-proposition}]
    Let $\alpha\in \R$, and consider $f\in H^\alpha(\Omega)$. Define the subspace $E:= L(C^\infty_c(\Omega))$ of $H^{1-2s-\alpha}_{scl}(\R^n)$, where $L:= (-\Delta)^s-\tau^{2s}r^{2s}+q$. Finally, define on $E$ the functional $T$ given by $$T(Lv):= \langle f,v \rangle, \qquad \mbox{ for all } v\in C^\infty_c(\Omega). $$
    Observe that $T$ is well-posed by the linearity of $L$ and the estimate in Proposition \ref{prop_apriori}. The same inequality also implies that $T$ is bounded, since $(H^\alpha_{scl}(\Omega))^*= H^{-\alpha}_{scl,\overline\Omega}(\R^n)$ and therefore
    \begin{align*}
        |T(Lv)| & = |\langle f,v \rangle| \leq \|f\|_{H^\alpha_{scl}(\Omega)}\|v\|_{(H^\alpha_{scl}(\Omega))^*} \lesssim \|f\|_{H^\alpha_{scl}(\Omega)} h^{2s-1} \|Lv\|_{H_{scl}^{1-2s-\alpha}(\R^n)}.
    \end{align*}
    By the Hahn-Banach theorem, it is possible to find an extension $\tilde T$ of $T$ which is a bounded functional on $H^{1-2s-\alpha}_{scl}(\R^n)$  with norm at most $\|f\|_{H^\alpha_{scl}(\Omega)} h^{2s-1}$. Using the Riesz representation theorem, we find the unique $\tilde u \in (H^{1-2s-\alpha}_{scl}(\R^n))^* = H^{2s+\alpha-1}_{scl}(\R^n)$ such that $$ \tilde T(w) = \langle \tilde u,w \rangle, \qquad \mbox{ for all } w\in H^{1-2s-\alpha}_{scl}(\R^n).$$
    If now $u:=\tilde u|_{\Omega}$, we have $$\|u\|_{H^{2s+\alpha-1}_{scl}(\Omega)} \leq \|\tilde u\|_{H^{2s+\alpha-1}_{scl}(\R^n)} \lesssim h^{2s-1} \|f\|_{H^\alpha_{scl}(\Omega)} , $$ and for all $v\in C^\infty_c(\Omega)$ we have
    $$ \langle f, v \rangle = T(Lv) = \tilde T(Lv) = \langle \tilde u, Lv \rangle  = \langle L\tilde u,v \rangle.  $$ 
\end{proof}

The proofs of Theorems \ref{thm_main1} and \ref{thm_main2} now follow by combining Proposition \ref{first-solvability-proposition} with each of the two cases of Proposition \ref{prop:approx-sol-rq}.

\section{Applications to inverse problems} \label{sec:applications}

We begin by considering the inverse problem for the operator \begin{equation}\label{op-scalar-case}
    (-\Delta)^s - \tau^{2s} r(x)^{2s} + q(x)
\end{equation}
consisting in recovering the refraction index $r\in L^\infty(\Omega)$ and potential $q\in L^\infty(\Omega)$ from DN data. The proof of well-posedness for the Dirichlet problem associated to the above operator in the space $H^s(\R^n)$ does not differ substantially to the one obtained in \cite{GSU16} for the case $r\equiv 0$. We shall always assume that $0$ is not a Dirichlet eigenvalue for said problem, which ensures the well-definedness of the DN map $\Lambda_{\tau,r,q}: H^s(\Omega_e)\rightarrow H^s(\Omega_e)^*$ via the associated bilinear form $B_{\tau,r,q}$: 
$$\langle \Lambda_{\tau,r,q}[f],[g] \rangle := B_{\tau,r,q}(u_f,g) = \langle (-\Delta)^{s/2} u_f, (-\Delta)^{s/2}g \rangle + \langle (q- \tau^{2s} r^{2s})r_\Omega u_f,r_{\Omega}g \rangle_{\Omega}. $$
Using the above definition, it is easy to obtain an Alessandrini identity. For $j=1,2$, let $f_j\in C^\infty_c(\Omega_e)$ be an exterior value, and assume that $u_j\in H^s(\R^n)$ is the corresponding solution in the Dirichlet problem for the operator \eqref{op-scalar-case} with coefficients $r_j, q_j \in L^\infty(\Omega)$. Let also $\Lambda_{\tau,j} := \Lambda_{\tau,r_j,q_j}$. Then it holds
\begin{equation}\label{alessandrini-scalar}
    \langle (\Lambda_{\tau,1}-\Lambda_{\tau,2})[f_1],[f_2] \rangle = \langle ((q_1 - q_2) + \tau^{2s} (r_2^{2s} - r_1^{2s}))r_\Omega u_1,r_{\Omega} u_2 \rangle_{\Omega},
\end{equation}
from which we can easily obtain a uniqueness result for fixed frequencies using the Runge approximation technique from \cite{GSU16}. However, the stability associated to this method is at most logarithmic. In this section, we will prove Theorem \ref{Th:stability-high-s}, thus showing that Hölder stability for the inverse problem in the regime $s\geq 1/2$ can be achieved by using high frequencies. We begin by rigorously defining the Cauchy data $C_{r,q_j}^{\tau}$ associated to the operator $$P_j := (-\Delta)^s - \tau^{2s} r^{2s} + q_j$$ and the distance function $\delta$ measuring the distance between the Cauchy data for the potentials $q_1, q_2$.

The operator $P_j$ is associated to the symmetric bilinear form $$ B_j(u,v):= ((-\Delta)^{s/2}u, (-\Delta)^{s/2}v)_{\R^n} - \tau^{2s} (r^{2s}u|_{\Omega},v|_{\Omega})_{\Omega} + (q_ju|_{\Omega},v|_{\Omega})_{\Omega},$$
defined for all $u,v\in H^s(\R^n)$. Let $S_j$ be the linear subspace of $H^s(\R^n)$ consisting of weak solutions to $P_ju=0$ in $\Omega$, that is 
$$ S_j := \{ u\in H^s(\R^n) : P_ju=0 \mbox{ in } \Omega \}, $$
and consider the linear operator $T_j : S_j \rightarrow X^*$ defined for all $v\in H^s(\R^n)$ by
$$ \langle T_ju, [v] \rangle:=  B_j(u,v). $$
We have to check that $T_j$ is actually well-defined. Recall that $X$ is given by $$X:= H^s(\R^n)/\widetilde H^s(\Omega),$$ 
to which we associate the usual quotient norm. Thus $T_j$ will be well-defined as soon as $B_j(u,v)=B_j(u,v+\varphi)$ for all $u\in S_j$, $v\in H^s(\R^n)$ and $\varphi \in \widetilde H^s(\Omega)$. This is true because $u\in S_j$ implies $B_j(u,\varphi)=0$ for all $\varphi\in \widetilde H^s(\Omega)$. The operator $T_j$ is bounded, since for all $\varphi\in \widetilde H^s(\Omega)$ we have
\begin{align*}
    \langle  T_ju, [v] \rangle = B_j(u,v+\varphi) \lesssim \|u\|_{H^s(\R^n)}\|v+\varphi\|_{H^s(\R^n)},
\end{align*}
and thus, by taking infimum with respect to $\varphi\in \widetilde H^s(\Omega)$ we get
\begin{align*}
    \|T_ju\|_{X^*} = \sup_{\|[v]\|_X=1} \langle  T_ju, [v] \rangle \lesssim \|u\|_{H^s(\R^n)}\sup_{\|[v]\|_X=1}\inf_{\varphi\in \widetilde H^s(\Omega)} \|v+\varphi\|_{H^s(\R^n)} = \|u\|_{H^s(\R^n)}.
\end{align*}

It is known that if $\Omega$ is Lipschitz, then $X$ can be naturally identified with $H^s(\Omega_e)$; however, recall that while objects in $X$ are equivalence classes with representatives in $H^s(\R^n)$, the elements of $H^s(\Omega_e)$ are restrictions of $H^s(\R^n)$ functions to $\Omega_e$. As a consequence, $X^*$ can be identified with $(H^s(\Omega_e))^*$, and therefore with $ H^{-s}_{\overline \Omega_e}(\R^n)$, which is the set of $H^s(\R^n)$ functions with support in $\overline\Omega_e$. Finally, only for $s<1/2$ we can identify $ H^{-s}_{\overline \Omega_e}(\R^n)$ with $H^{-s}(\Omega_e)$. With these distinctions in mind, we define the following Cauchy data
$$C_{r,q_j}^{\tau}:=\{ ([u], T_ju) \in X\times X^* : u\in S_j \}$$
and the distance
\begin{align*}
    \delta(C_{r,q_1}^{\tau},C_{r,q_2}^{\tau}) :&= \max_{j\neq k} \sup_{\substack{
    \{u\in S_k : \| [u] \|_X=1\}} 
} \inf_{ v \in S_j} \| ([v], T_jv) - ([u], T_ku) \|_{X\times X^*} 
\\ & = \max_{j\neq k} \sup_{\substack{
    \{u\in S_k : \| [u] \|_X=1\}} 
} \inf_{ v \in S_j} \left(\| [v-u] \|_{X} + \| T_jv - T_ku \|_{X^*} \right) .
\end{align*} 
Observe that when $\Omega$ is Lipschitz both spaces $X, X^*$ can be identified with spaces of functions on $\Omega_e$, and thus the Cauchy data given above is exterior. If $s<1/2$, both the Cauchy data and the distance function coincide with the usual ones. \\

We can now turn to the proof of Theorem \ref{Th:stability-high-s}.

\begin{proof}[Proof of Theorem \ref{Th:stability-high-s}]
\textbf{Step 1.} Using the symbols introduced in the above discussion, let $u_1 \in S_1$ and $u_2,\tilde u_2\in S_2$. We compute
\begin{align*}
 ((q_1-q_2) u_1|_{\Omega}, u_2|_{\Omega})_{\Omega} & = B_1(u_1,u_2) - B_2(u_1,u_2) 
 \\ & = B_1(u_1,u_2) - B_2(\tilde u_2,u_2) + B_2(u_2,\tilde u_2-u_1) 
 \\ & = \langle T_1u_1- T_2\tilde u_2, [u_2]\rangle +\langle T_2u_2,[\tilde u_2-u_1]\rangle 
 \\ & \leq 
 \|T_1u_1- T_2\tilde u_2\|_{X^*} \|[u_2]\|_X + \|T_2u_2\|_{X^*}\|[\tilde u_2-u_1]\|_X.
 \end{align*} 
 By the boundedness of $T_2$ and the well-posedness estimate for the problem relative to operator $P_2$, we have for all $\varphi\in \widetilde H^s(\Omega)$
 $$ \|T_2u_2\|_{X^*}\lesssim \|u_2\|_{H^s(\R^n)} \lesssim  \|u_2+\varphi\|_{H^s(\R^n)}, $$
 and thus by taking the infimum with respect to $\varphi$ we get $\|T_2u_2\|_{X^*}\lesssim \|[u_2]\|_{X}$. Eventually, we obtain the estimate
 \begin{align*}
 ((q_1-q_2) u_1|_{\Omega}, u_2|_{\Omega})_{\Omega} &\lesssim  
 \left(\|T_1u_1- T_2\tilde u_2\|_{X^*} + \|[\tilde u_2-u_1]\|_X\right)\|[u_2]\|_{X}
 \\ & =  
 \left(\|T_1v_1- T_2\tilde v_2\|_{X^*} + \|[\tilde v_2-v_1]\|_X\right)\|[u_1]\|_{X}\|[u_2]\|_{X},
 \end{align*}
 where we normalized $v_1 := u_1/\|[u_1]\|_{X}$ and $\tilde v_2 := \tilde u_2/\|[u_1]\|_{X}$, in order to have $ \|[v_1]\|_{X}=1$. Taking on both sides the infimum with respect to $\tilde v_2\in S_2$ and then the supremum with respect to $v_1\in S_1$ with $ \|[v_1]\|_{X}=1$, we finally get
\begin{align*}
|((q_1-q_2) u_1|_{\Omega}, u_2|_{\Omega})_{\Omega}| \lesssim \delta(C_{r,q_1}^{\tau},C_{r,q_2}^{\tau})\|[u_1]\|_{X}\|[u_2]\|_{X} .
\end{align*}
If $\Omega$ is Lipschitz, we can alternatively restate the above inequality as
\begin{equation}\label{stability-th-eq:1}
|((q_1-q_2) u_1|_{\Omega}, u_2|_{\Omega})_{\Omega}| \lesssim \delta(C_{r,q_1}^{\tau},C_{r,q_2}^{\tau})\|u_1|_{\Omega_e}\|_{H^s(\Omega_e)}\|u_2|_{\Omega_e}\|_{H^s(\Omega_e)} .
\end{equation}

\textbf{Step 2.} Next, we will show that $\norm{u_j}_{H^s(\Omega_e)} \leq C \tau^{s}$. To this end, let $M\in\N$, and let $u_M:= e^{i\tau\varphi_0} a$ be the approximate solution constructed in Section \ref{Sec:approx}. The phase $\varphi_0$ and the amplitude $a$ are obtained by solving eikonal and transport equations near $\overline\Omega$. Given an open set $\Omega'\supset \overline\Omega$, define a smooth cutoff function $\chi\in C^\infty_c(\Omega')$ that equals $1$ near $\overline\Omega$, and let $u'_M:= \chi u_M = e^{i\tau\varphi_0} \chi a$. After a normalization we can fix $\|u'_M\|_{L^2(\R^n)}=1$, which implies
$$ \|\nabla u'_M\|_{L^2(\R^n)} \leq  \tau \|\nabla\varphi_0 u'_M \|_{L^2(\R^n)} + O(1) = O(\tau),$$
and therefore $\|u'_M\|_{H^1(\R^n)} = O(\tau)$. It now follows by interpolation that
$$ \|u'_M\|_{H^s(\R^n)} \leq \|u'_M\|_{L^2(\R^n)}^{1-s}\|u'_M\|_{H^1(\R^n)}^s = O(\tau^s). $$
The approximate solution $u'_M$ can be upgraded to a solution $u = u'_M + r_M$ to the problem $Lu=0$ in $\Omega$, where by Proposition \ref{first-solvability-proposition} for all $\alpha\in\R$ the error $r_M$ verifies
$$ \|r_M\|_{H^{2s-1+\alpha}_{scl}(\R^n)} \lesssim \tau^{1-2s}\|Lu'_M\|_{H^\alpha_{scl}(\Omega)}.$$
Since by fixing $\alpha = 1-s$ we can compute for large $\tau$
\begin{align*}
    \|r_M\|_{H^s(\R^n)} & \leq \tau^{s}\|r_M\|_{H^s_{scl}(\R^n)}
    \\ & \lesssim
    \tau^{1-s}\|Lu'_M\|_{H^{1-s}_{scl}(\Omega)}
    \\ & \leq 
    \tau^{1-s}\|Lu'_M\|_{L^2(\Omega)}^{s}\|Lu'_M\|_{H^1_{scl}(\Omega)}^{1-s},
\end{align*}
in order to obtain the wanted estimate for $\|u\|_{H^s(\R^n)}$ we only need to estimate $Lu'_M = L(\chi u_M)$ in the norm of $H^{\beta}_{scl}(\Omega)$, for $\beta=1,2$. Recall that we have already obtained similar estimates for $Lu_M$ as part of the construction of geometrical optics approximate solutions. These were derived from formula \eqref{after-eikonal}, which for $u'_M$ implies
\begin{equation*}\begin{split}
& e^{-i\tau\varphi}L u'_M =  
 \sum_{\mu=1}^{N}  \sum_{\substack{ \{ t_j \}_{j=1}^{F_M} \subset \N: \\ \sum_j t_j\leq \mu  }} \tau^{2s-\mu+\sum_{j=1}^{F_M} t_{j}(1-2sj)} L_{\mu;\mathcal T'}(\chi a)  + q \chi a + O(\tau^{-n-N/2}).
\end{split}\end{equation*}
We obtained the components $a_l$ of the amplitude $a$ by solving transport equations close to $\overline\Omega$ for all the distinct powers of $\tau$, up to a fixed $\tau^{-A_M}$. Given the definition of the cutoff function $\chi$ and the fact that all the operators $L_{\mu;\mathcal T'}$ are local, we see that for $x\in\Omega$
$$Lu'_M = Lu_M + O(\tau^{-n-N/2}),$$
and thus we have the same estimates for $\|Lu'_M\|_{H^\beta_{scl}(\Omega)}$ as for $\|Lu_M\|_{H^\beta_{scl}(\Omega)}$. Eventually
\begin{align*}
\|u\|_{H^s(\R^n)} & \leq \|u'_M\|_{H^s(\R^n)} + \|r_M\|_{H^s(\R^n)}
\\ & \lesssim
\tau^{1-s} \|Lu'_M\|_{L^2(\Omega)}^{s}\|Lu'_M\|_{H^1_{scl}(\Omega)}^{1-s} + O(\tau^s)
\\ & \lesssim
\tau^{1-s} \|Lu_M\|_{L^2(\Omega)}^{s}\|Lu_M\|_{H^1_{scl}(\Omega)}^{1-s} + O(\tau^s)
\\ & \lesssim
\tau^{1-s-M} + O(\tau^s) = O(\tau^s).
\end{align*}

The argument above holds for both $u_1$ and $u_2$, so we have obtained in particular $\norm{u_j}_{H^s(\Omega_e)} \lesssim \tau^{s}$ for $j=1,2$. Coming back to equation \eqref{stability-th-eq:1}, we get that the inequality
\begin{equation}\label{stability-th-eq:2}
\int_\Omega (q_1-q_2)u_1\overline{u_2} \,dx  \lesssim  \left[ \sup_{\tau \geq \tau_0} \delta(C_{r,q_1}^{\tau}, C_{r,q_2}^{\tau}) \right] \tau^{2s}
\end{equation}
holds for all $\tau\geq \tau_0$. For simplicity of notation, we will let $\delta:= \sup_{\tau \geq \tau_0} \delta(C_{r,q_1}^{\tau}, C_{r,q_2}^{\tau})$. \\

\textbf{Step 3.} We now argue as in Section 4 of \cite{MSS22}. Consider the simple manifold  $(\Omega,g)$, where the metric $g$ is given by $$ g|_{(x_1, ..., x_n)} := r(x)^{-2}\delta_{ij} dx^i\otimes dx^j, \qquad |g|:= \det g_{ij} = r(x)^{-2n}, $$
$$ \langle X,Y\rangle_g := r(x)^{-2} X\cdot Y, \qquad |X|_g := r(x)^{-1} |X|. $$  We can choose an open set $\Omega'\supset \overline\Omega$ such that $(\Omega',g)$ is also simple (recall that $r\equiv 1$ outside of $\Omega$). If $p$ is any point such that $p\in\p\Omega'$, by the properties of simple manifolds each $x\in \Omega$ can be expressed as $x=\exp_p(\rho\theta)$, where $\rho>0$ and $\theta\in S^{n-1}$. This defines the \emph{polar normal coordinates} on $\Omega$, which allow us to write the metric $g$ as
$$ g|_{(\rho,\theta)} = d\rho^2 + g_0(\rho,\theta)d\theta \otimes d\theta, \qquad |g|=|g_0|.$$
We observe that the functions $\varphi_0(\rho,\theta)=\pm\rho$ solve the eikonal equation. In fact, since the $\varphi_0$ defined above do not depend on $\theta$, by the expression of the metric in polar normal coordinates we get
$$ r(x)^{-2}\nabla\varphi_0 \cdot \nabla \psi = \langle\nabla\varphi_0, \nabla \psi\rangle_g = \pm \p_\rho\psi, \qquad  |\nabla\varphi_0| = r(x)|\nabla\varphi_0|_g = r(x). $$ \\

\textbf{Step 4.} Having solved the eikonal equation and found the phase function $\varphi_0$, we now turn to the transport equation for the principal amplitude $a_0$. We see that $a_0$ solves
\begin{equation*}
    \begin{cases}
        L_{1;0}a_0=0, & \mbox{ if }  s\in (1/2,1), \\
        (L_{1;0}+q)a_0=0, & \mbox{ if }  s = 1/2,
    \end{cases}
\end{equation*}
and thus we need to compute the operator $L_{1;0}$. Let us write the Laplace-Beltrami operator $\Delta_g$ in Euclidean and polar normal coordinates. We have
\begin{align*}
    \Delta_g \varphi_0 & = |g|^{-1/2}\p_i(|g|^{1/2}g^{ij}\p_j \varphi_0)  = r^{n}\nabla\cdot(r^{2-n}\nabla \varphi_0) \\ & =  r^{2}\Delta \varphi_0 + (2-n)r\nabla r\cdot\nabla \varphi_0 = r^{2}\Delta \varphi_0 \pm (2-n)r^3\p_\rho r
\end{align*}
and 
$$ \Delta_g \varphi_0 = |g|^{-1/2}\p_\rho ( |g|^{1/2}\p_\rho \varphi_0 )=  \pm r^{n}\p_\rho ( r^{-n} ) = \mp n r^{-1}\p_\rho r,   $$
so that 
$$ r^{2}\Delta \varphi_0  = \mp \left( n r^{-1} + (2-n)r^3 \right) \p_\rho r, $$
or equivalently
$$ \mp r^{-2}\Delta \varphi_0  = \left( n r^{-5} + (2-n)r^{-1} \right) \p_\rho r. $$
As observed in the Appendix, we can write the operator $L_{1;0}$ as
\begin{align*}
L_{1;0} =\pm 2si r^{2s} \left( \p_\rho \mp \frac{r^{-2}\Delta\varphi_0 }{2}+(1-s) r^{-1}\p_\rho r\right),
\end{align*}
where the sign changes are due to the fact that $L_{1;0}$ depends on the phase function $\varphi_0$. Thus we get
\begin{align*}
L_{1;0} &=  
    \pm 2si r^{2s} \left( \p_\rho +   \left( \frac{n}{2} r^{-5} + \left(2-\frac{n}{2}-s\right)r^{-1} \right) \p_\rho r \right) =
    \pm 2si r^{2s}\left( \p_\rho +  \p_\rho f \right),
    \end{align*}
where $f(r):=  \left(2-\frac{n}{2}-s\right)\log r -\frac{n}{8} r^{-4}$, and the principal amplitude $a_0$ must solve
\begin{equation*}
    \begin{cases}
        \p_\rho a_0 +  a_0 \p_\rho f =0, & \mbox{ if }  s\in (1/2,1), \\
        \p_\rho a_0 +  a_0\left(\p_\rho f \mp i q r^{-1}\right) =0, & \mbox{ if }  s = 1/2.
    \end{cases}
\end{equation*}
Therefore, we obtain $a_0(\rho,\theta) = b(\theta)e^{-f(\rho,\theta)+iJ(q(\rho,\theta))}$, where $b$ is any function depending only on $\theta$, and
\begin{equation*}
  J(q(\rho,\theta)) := \begin{cases}
        0, & \mbox{ if }  s\in (1/2,1), \\
       \pm \int_0^\rho q(t,\theta) r(t,\theta)^{-1}dt, & \mbox{ if }  s = 1/2.
    \end{cases}
\end{equation*}
In polar normal coordinates, an approximate solution (of order $M\in\N$) corresponding to a fixed potential $q$ will take the form
$$ u_M(\rho,\theta) = 
 e^{\pm i\tau \rho} a_{0} + v_M, $$
where the error $v_M$ depending on $b$ and $q$ is defined as  $$v_M :=  e^{i\tau\varphi_0}\sum_{l=1}^{A_M} \tau^{-\alpha_l} a_{l}.$$
It also follows from the construction of approximate geometrical optics solutions that for all $j\in\N$ the amplitude $a_{j+1}$ is obtained by solving the transport equation $$ L_{1;0}a_{j+1} = \sum_{i=0}^j \tilde L_{ij} (a_i),$$
where the differential operator $\tilde L_{ij}$ has order at most $N$, and does not depend on the function $b$. Thus we can estimate $\|a_{j+1}\|_{H^\beta} \lesssim \sum_{i=0}^j \|a_i\|_{H^{\beta+N}}$, and ultimately using induction we obtain that there exists a $t_M\in\N$ so large that $$ \|a_{j}\|_{L^2} \lesssim \|a_0\|_{H^{t_M}} , \qquad \mbox{for all } j\in\{ 1, ..., A_M\}.$$
This allows us to estimate
$$ \|v_M\|_{L^2} \lesssim \tau^{-\alpha_1} \|a_0\|_{H^{t_M}} , $$
where the implied factor depends on $M$, but not on the function $b$. The exponent $\alpha_1$ in the above formula is as in the statement of Proposition \ref{prop:approx-sol-rq}, and thus it is the smallest real number in the set $(\N + (2s-\floor{2s})\N)\setminus \{0\}$. As such, we have 
$$ \alpha_1 = \begin{cases}
    2s-1 & \mbox{ if } s\in (1/2,1), \\
    1 & \mbox{ if } s=1/2.
\end{cases}$$\\

\textbf{Step 5.} Let now $u_1$ be the exact solution of equation $((-\Delta)^s - \tau^{2s} r^{2s} + q_1) u_1 = 0$ given by $$ u_1 = u_{M,1} + r_{M,1} = e^{\pm i\tau \rho } b_1 e^{-f+iJ(q_1)} + v_{M,1} + r_{M,1}, $$
where $b_1, q_1$ are a fixed smooth function and potential. We are indicating by $r_{M,1}$ the error due to the approximation to order $M\in \N$ from approximate geometrical optics to exact solution, while by $v_{M,1}$ we are indicating the error due to the non-principal amplitudes, as defined above. We can do a similar construction for $u_2$, which corresponds to potential $q_2$ and function $b_2$. This gives
\begin{align*}
     \left|\int_\Omega (q_1-q_2)u_{M,1}\overline{u_{M,2}} dx \right| \leq  \left|\int_\Omega (q_1-q_2)u_1\overline{u_2} dx \right| + O(\tau^{1-M}),
\end{align*}
where the small error of order $\tau^{1-M}$ is independent of the functions $b_1, b_2$ and is due to the error terms $r_{M,1}, r_{M,2}$, which we estimated as above. We also compute
\begin{align*}
    \left|\int_\Omega (q_1-q_2) e^{-2f+iJ(q_1-q_2)} b_1 b_2 \; dx\right| & \leq \left| \int_\Omega (q_1-q_2)u_{M,1}\overline{u_{M,2}} dx\right| \\ & \quad + C\tau^{-\alpha_1}\|a_{0,1}\|_{H^{t_M}}\|a_{0,2}\|_{H^{t_M}},
\end{align*}
where the error term of order $\tau^{-\alpha_1}$ is obtained from the estimates for $v_{M,1}, v_{M,2}$ and the constant $C$ is independent of $b_1, b_2$. \\

For the sake of simplicity, let us define
\begin{align*}
    Q:= (q_1-q_2) e^{-2f+iJ(q_1-q_2)}r^{-n}.
\end{align*}
Using the last two inequalities and \eqref{stability-th-eq:2}, we now obtain the estimate
\begin{align*}
    \left|\int_\Omega Q r^{n} b_1 b_2 \; dx\right| & \lesssim \delta \tau^{2s} + \tau^{-\alpha_1}\|a_{0,1}\|_{H^{t_M}}\|a_{0,2}\|_{H^{t_M}},
\end{align*} 
where the implied constant is independent of $b_1, b_2$.
Since $dx = |g|^{1/2}d\rho d\theta = r^{-n} d\rho d\theta $, using the fact that $q_1=q_2$ in $\R^n\setminus\Omega$ we have
\begin{equation}
    \label{geodesic-x-ray-transform}
    \int_\Omega Q r^{n} b_1 b_2 \; dx = \int_{\p_+S_p\Omega'} b_1(\theta) b_2(\theta)\int_0^{\tau_{\Omega'}(p,\theta)} Q(\rho,\theta) d\rho d\theta = \int_{\p_+S_p\Omega'} b_1 b_2 IQ(p,\cdot) d\theta,
\end{equation}
where $I$ is the geodesic ray transform on $\Omega'$. This gives us the estimate

\begin{align*}
    \left|\int_{\p_+S_p\Omega'} b_1 b_2 IQ(p,\cdot) d\theta\right| & \lesssim \delta \tau^{2s} + \tau^{-\alpha_1}\|a_{0,1}\|_{H^{t_M}}\|a_{0,2}\|_{H^{t_M}},
\end{align*}
which holds for every choice of $b_1,b_2$ with a uniform constant independent of the point $p\in\p\Omega'$. Therefore, integrating over $\p\Omega'$ gives
\begin{align*}
    |\int_{\p_+S\Omega'} &b_1(\theta) b_2(\theta)  IQ(p,\theta) d(\p S\Omega)| \leq \int_{\p\Omega'} |\int_{\p_+S_p\Omega'} b_1(\theta) b_2(\theta) IQ(p,\theta) d\theta | dp
    \\ & \lesssim
   \delta \tau^{2s} + \tau^{-\alpha_1}\int_{\p\Omega'} \|a_{0,1}\|_{H^{t_M}}\|a_{0,2}\|_{H^{t_M}}dp
    \\ & \lesssim
   \delta \tau^{2s} + \tau^{-\alpha_1}\int_{\p\Omega'} \|b_{1}\|_{H^{t_M}(\p_+S_p\Omega')}\|b_{2}\|_{H^{t_M}(\p_+S_p\Omega')}dp.
\end{align*}
 Let us choose $b_1(\theta) := I(I^*IQ)\nu_p\cdot\theta$, where $\nu_p$ is the normal to $\p\Omega'$ at $p$, and $b_2(\theta):=1$. By the boundedness of $I, I^*I$, this gives
\begin{align*}
    \| I^*IQ \|_{L^2(\Omega')}^2 & \lesssim 
    \delta \tau^{2s} + \tau^{-\alpha_1}\int_{\p\Omega'} \|I(I^*IQ)\|_{H^{t_M}(\p_+S_p\Omega')}dp
    \\ & \lesssim
    \delta \tau^{2s} + \tau^{-\alpha_1} \|Q\|_{H^{t_M}(\Omega)}.
\end{align*}
Using a stability estimate for $I^*I$ on simple manifolds, we get  
\begin{align*}
    \|Q\|^2_{H^{-1}} \lesssim \| I^*IQ \|_{L^2(\Omega')}^2 &\lesssim
   \delta \tau^{2s} + \tau^{-\alpha_1}\|Q\|_{H^{t_M}(\Omega)}.
\end{align*}
Using the fact that $r$ is bounded away from $0$ and the interpolation inequality, we obtain
\begin{align*}
    \|q_1-q_2\|^{2+\frac{2}{t_M} }_{L^2}& \lesssim \| Q \|^{\frac{2(t_M+1)}{t_M}}_{L^2} \leq \|Q\|^{2}_{H^{-1}}\|Q\|^{\frac{2}{t_M}}_{H^{t_M}}\lesssim 
    \left(\delta \tau^{2s} + \tau^{-\alpha_1}\|Q\|_{H^{t_M}}\right)\|Q\|^{\frac{2}{t_M}}_{H^{t_M}} .
\end{align*}
Observe now that $Q$ is independent of $\tau$, and therefore using the high order Sobolev estimates for $q_1,q_2$ we get $\|Q\|_{H^{t_M}}\lesssim 1$. As a result,
\begin{align*}
    \|q_1-q_2\|^{2+\frac{2}{t_M} }_{L^2}& \lesssim  \delta \tau^{2s} + \tau^{-\alpha_1} .
\end{align*}
Because $\alpha_1>0$, the above inequality can be optimized with respect to $\tau$: choosing $\tau = \delta^{-\frac{1}{2s+\alpha_1}}$, we finally get the stability estimate
\begin{align*}
    \|q_1-q_2\|_{L^2}& \lesssim  \delta^{\frac{\alpha_1}{4s+2\alpha_1}\frac{t_M}{t_M+1}} .
\end{align*}
Here $t_M$ can be taken arbitrarily large, and so the exponent $\gamma$ can approach $\frac{\alpha_1}{4s+2\alpha_1}$. However, this implies that we need to assume higher order Sobolev estimates for $q_1,q_2$.
\end{proof}

\begin{Rem}\label{rem:why-not}
    As discussed in Remark \ref{rem:why-small-s-hard}, the geometrical optics solutions constructed in Proposition \ref{prop:approx-sol-rq} for the two regimes $s\in (0,\frac 12)$ and $s\in[\frac 12,1)$ differ substantially. Having now completely solved the eikonal and transport equations in polar normal coordinates, we see that in both cases the principal phase function $\varphi_0$ oscillates very rapidly for large $\tau$ in the direction of the geodesics of the simple refraction index $r$, while it is constant in the orthogonal direction $\theta$. If there are no other phase functions, which is the case for $s\in[\frac 12,1)$, this behaviour allows us to interpret the integral $\int_\Omega (q_1-q_2)u_1\overline{u_2}dx$ in terms of the geodesic ray transform, as in equation \eqref{geodesic-x-ray-transform}. This in turn is our main tool in the final Step 5 of the proof.\\
    
    However, in the regime $s\in (0,\frac 12)$ the geometrical optics construction produces a sequence of non vanishing phase functions $\varphi_j$, associated to both positive and negative powers of $\tau$. These additional phase functions do not follow the geodesics of $r$, and also they do not necessarily oscillate rapidly for $\tau\rightarrow\infty$, which makes our method inapplicable in the case $s\in (0,\frac 12)$. Heuristically, one can reason that geometrical optics solutions with "classical" behaviour only arise when the fractional parameter $s$ is close enough to $1$ (and thus the main part of the operator is close to the Laplacian), while for $s <\frac 12$ they do not exist due to the fact that the main part of the operator has order less than $1$, and is thus very different from the Laplacian. 
\end{Rem}

\section{Appendix}

In the proof of the stability result of Theorem \ref{Th:stability-high-s} we need to know the exact expression of the operator $L_{1;0}$ in polar normal coordinates. This is computed in the following Lemma.

\begin{Lem}
    Let $L_{1;0}$ be the first order operator defined by \eqref{def-operators-L} for $\mu=1, \mathcal T'=0$. Moreover, let $(\rho,\theta)$ indicate the normal polar coordinates relative to the simple refraction index $r$, as defined in Step 3 of the proof of Theorem \ref{Th:stability-high-s}. Then it holds
    $$ L_{1;0} = \pm 2si r^{2s} \left( \p_\rho \mp \frac{r^{-2}\Delta\varphi_0 }{2}+(1-s) r^{-1}\p_\rho r\right), $$
    where the sign changes are due to the choice of phase function $\varphi_0=\pm\rho$.
\end{Lem}

\begin{proof}
    The proof is a long, but elementary computation. We start from the definition of the operator, which is given by formula \eqref{def-operators-L}:
\begin{align*}
    L_{1;0} &= \sum_{ |\delta|\leq 1} D^{\delta} \sum_{t_0=0}^{N-1} \frac{i^{1+t_0}}{({1+t_0})!}\sum_{\substack{\beta\geq\delta \\ |\beta|={1+t_0}}}  D^\beta_\xi(|\xi|^{2s})|_{\xi = \nabla \varphi_0}    \sum_{\substack{ \{ \eta_j \}_{j=0}^{F_M} \subset \N^n: \\ \sum_{j=0}^{F_M} \eta_j = \beta-\delta, \\ |\eta_0|\geq t_0}}   R^{\beta,\delta, \mathcal N}_{t_0, 0,...,0}
    \\ & = 
     \sum_{t_0=0}^{N-1} \frac{i^{1+t_0}}{({1+t_0})!}\sum_{ |\beta|=1+t_0} D^\beta_\xi(|\xi|^{2s})|_{\xi = \nabla \varphi_0}    \sum_{\substack{ \{ \eta_j \}_{j=0}^{F_M} \subset \N^n: \\ \sum_{j=0}^{F_M} \eta_j = \beta, \\ |\eta_0|\geq t_0}}   R^{\beta,0, \mathcal N}_{(t_0,0)}
    \\ & \quad +
    \sum_{ |\delta|= 1} D^{\delta} \sum_{t_0=0}^{N-1} \frac{i^{1+t_0}}{({1+t_0})!}\sum_{\substack{\beta\geq\delta \\ |\beta|={1+t_0}}}  D^\beta_\xi(|\xi|^{2s})|_{\xi = \nabla \varphi_0}    \sum_{\substack{ \{ \eta_j \}_{j=0}^{F_M} \subset \N^n: \\ \sum_{j=0}^{F_M} \eta_j = \beta-\delta, \\ |\eta_0|\geq t_0}}   R^{\beta,\delta, \mathcal N}_{(t_0, 0)}
      \\ & = 
     \sum_{t_0=0}^{N-1} \frac{i^{1+t_0}}{({1+t_0})!}\sum_{ |\beta|=1+t_0} D^\beta_\xi(|\xi|^{2s})|_{\xi = \nabla \varphi_0}   \left( \sum_{ \substack{|\eta_0|= t_0 \\ \eta_0\leq\beta}} \sum_{\substack{A\in \N^{n\times (F_M-1)} : \\|A|=1}} R^{\beta,0, (\eta_0,A)}_{(t_0,0)}
    +  \sum_{ \substack{|\eta_0|= 1+t_0\\ \eta_0\leq\beta}}   R^{\beta,0, (\eta_0,0)}_{(t_0,0)}  \right)
    \\ & \quad +
    \sum_{ |\delta|= 1} D^{\delta} \sum_{t_0=0}^{N-1} \frac{i^{1+t_0}}{({1+t_0})!}\sum_{\substack{\beta\geq\delta \\ |\beta|={1+t_0}}}  D^\beta_\xi(|\xi|^{2s})|_{\xi = \nabla \varphi_0}    \sum_{\substack{|\eta_0|= t_0 \\ \eta_0\leq\beta}}   R^{\beta,\delta, (\eta_0, 0)}_{(t_0,0)}
      \\ & = 
     \sum_{t_0=0}^{N-1} \frac{i^{1+t_0}}{({1+t_0})!}\sum_{ |\sigma|=1+t_0} D^\sigma_\xi(|\xi|^{2s})|_{\xi = \nabla \varphi_0}   \left( \sum_{ \substack{|\eta_0|= t_0, \\ \eta_0\leq\sigma}} \sum_{\substack{A\in \N^{n\times (F_M-1)} : \\|A|=1}} R^{\sigma,0, (\eta_0,A)}_{(t_0,0)}
    +  \sum_{ \substack{|\eta_0|= 1+t_0\\ \eta_0\leq\sigma}}   R^{\sigma,0, (\eta_0,0)}_{(t_0,0)}  \right)
    \\ & \quad +
    \sum_{ |\delta|= 1} D^{\delta} \sum_{t_0=0}^{N-1} \frac{i^{1+t_0}}{({1+t_0})!}\sum_{|\sigma|=t_0 }  D^{\sigma +\delta}_\xi(|\xi|^{2s})|_{\xi = \nabla \varphi_0}    \sum_{\substack{|\eta_0|= t_0\\ \eta_0\leq\sigma+\delta}}   R^{\sigma +\delta,\delta, (\eta_0, 0)}_{(t_0,0)}.
\end{align*} 
We now substitute the definition of the $R$ terms. These are
$$ R^{\sigma + \delta,\delta, \mathcal N}_{(t_0,0)}(y) = \binom{\sigma + \delta}{\delta}\binom{\sigma}{\sigma-\eta_0} I_{\eta_0, t_0, r_y}(0) \prod_{j=1}^{F_M} \binom{\sigma -\sum_{i=0}^{j-1}\eta_i}{\sigma -\sum_{i=0}^j\eta_i }  I_{\eta_{j},0,\varphi_{j}(\cdot+y)}(0),$$
with $r_{y}(x) = \varphi_0(y+x)-\varphi_0(y) - \nabla \varphi_0(y) \cdot x$ and
\begin{align*}
     I_{\alpha,m,v}(x) = i^m \sum_{ \substack{\rho \in (\N_{>0}^n)^m :\\  \sum_{j=1}^m\rho_j=\alpha}}\prod_{j=1}^m D^{\rho_j}_xv(x).
\end{align*}
In particular $r_{y}(0) = 0$, $\nabla r_{y}(0) = 0$, $\nabla^2 r_{y}(0) = \nabla^2 \varphi_0(y)$, and so 
\begin{align*}
      & I_{0,0,v}(0) = 1 , \qquad  I_{e_k,0,v}(0) = 0, 
      \\ & I_{\eta_0, |\eta_0|, r_y}(0) = i^{|\eta_0|} \sum_{ \substack{\rho \in (\N_{>0}^n)^{|\eta_0|} :\\  \sum_{j=1}^{|\eta_0|}\rho_j=\eta_0}}\prod_{j=1}^{|\eta_0|} D^{\rho_j}_xr_y(0) = i^{|\eta_0|} \sum_{ \substack{\rho \in (\N_{>0}^n)^{|\eta_0|} :\\  |\rho_j|=1}}\prod_{j=1}^{|\eta_0|} D^{\rho_j}_xr_y(0) = 0, \qquad \mbox{if } |\eta_0|\neq 0,
      \\ & I_{\eta_0, |\eta_0|-1, r_y}(0) = i^{|\eta_0|-1} \sum_{ \substack{\rho \in (\N_{>0}^n)^{|\eta_0|-1} :\\  \sum_{j=1}^{|\eta_0|-1}\rho_j=\eta_0}}\prod_{j=1}^{|\eta_0|-1} D^{\rho_j}_xr_y(0) =0, \qquad \mbox{if } |\eta_0|\neq 1,2,
       \\ & I_{\eta_0, 1, r_y}(0) = i D^{\eta_0}_xr_y(0), \qquad \mbox{if } |\eta_0|=2.
\end{align*} 
Therefore

$$ R^{\sigma,0, (\eta_0,A)}_{(t_0,0)}(y) = 0, \qquad \mbox{if } |A|=1,$$
$$ R^{\sigma,0, (\eta_0,0)}_{(|\eta_0|-1,0)}(y) = \binom{\sigma}{\sigma-\eta_0} I_{\eta_0, |\eta_0|-1, r_y}(0) = \begin{cases}
    0, & |\eta_0|\neq 2
    \\ i \binom{\sigma}{\sigma-\eta_0} D^{\eta_0}_xr_y(0), & |\eta_0|=2 \end{cases} ,$$
$$  R^{\sigma +\delta,\delta, (\eta_0, 0)}_{(|\eta_0|,0)} = \binom{\sigma + \delta}{\delta}\binom{\sigma}{\sigma-\eta_0} I_{\eta_0, |\eta_0|, r_y}(0) = \begin{cases}
    0, & |\eta_0|\neq 0
    \\ \binom{\sigma + \delta}{\delta}, & |\eta_0|=0 \end{cases} ,$$     
which gives
\begin{align*}
L_{1;0} &= i\sum_{ |\delta|= 1}   D^{\delta}_\xi(|\xi|^{2s})|_{\xi = \nabla \varphi_0} D^{\delta}  - 
  \frac{i }{2}\sum_{|\sigma|=2 }  D^{\sigma}_\xi(|\xi|^{2s})|_{\xi = \nabla \varphi_0}    D^\sigma \varphi_0
  \\ & =
  i  \nabla_\xi(|\xi|^{2s}) |_{\xi=\nabla\varphi_0}\cdot\nabla - \frac{i }{2} \nabla^2\varphi_0:\nabla^2_\xi (|\xi|^{2s}) |_{\xi=\nabla\varphi_0}
  \\ & =  2si |\nabla\varphi_0|^{2s-2} \nabla\varphi_0  \cdot\nabla - is |\nabla\varphi_0|^{2s-2}\nabla^2\varphi_0: ((2s-2)|\nabla\varphi_0|^{-2} \nabla\varphi_0\otimes\nabla\varphi_0 + I) 
    \\ & =  si |\nabla\varphi_0|^{2s-2} \left(2\nabla\varphi_0  \cdot\nabla - \Delta\varphi_0 \right) - si(2s-2) |\nabla\varphi_0|^{2s-4}(\nabla^2\varphi_0\cdot  \nabla\varphi_0)\cdot\nabla\varphi_0
    \\ & =  
    si |\nabla\varphi_0|^{2s-2} \left( \left(2\nabla\varphi_0  \cdot\nabla - \Delta\varphi_0 \right) + (2-2s) |\nabla\varphi_0|^{-1} \nabla\varphi_0\cdot  \nabla(|\nabla \varphi_0| )\right)
    \\ & =  
    \pm 2si r^{2s} \left( \p_\rho \mp \frac{r^{-2}\Delta\varphi_0 }{2}+(1-s) r^{-1}\p_\rho r\right).
    \end{align*}
\end{proof}

\bibliographystyle{plainnat}
\bibliography{ARXIV}

\begin{thebibliography}{57}
\providecommand{\natexlab}[1]{#1}
\providecommand{\url}[1]{\texttt{#1}}
\expandafter\ifx\csname urlstyle\endcsname\relax
  \providecommand{\doi}[1]{doi: #1}\else
  \providecommand{\doi}{doi: \begingroup \urlstyle{rm}\Url}\fi

\bibitem[Askes and Aifantis(2011)]{AskesAifantis}
H.~Askes and E.~Aifantis.
\newblock Gradient elasticity in statics and dynamics: an overview of
  formulations, length scale identification procedures, finite element
  implementations and new results.
\newblock \emph{International Journal of Solids and Structures}, 48\penalty0
  (13):\penalty0 1962--1990, 2011.

\bibitem[Baers et~al.(2024)Baers, Covi, and Rüland]{BCR24}
H.~Baers, G.~Covi, and A.~Rüland.
\newblock On instability properties of the fractional calderón problem.
\newblock \emph{Preprint arXiv:2405.08381}, 2024.

\bibitem[Banerjee and Senapati(2024)]{BS22}
A.~Banerjee and S.~Senapati.
\newblock The calderón problem for space-time fractional parabolic operators
  with variable coefficients.
\newblock \emph{SIAM Journal of Mathematical Analysis}, 56\penalty0 (4), 2024.

\bibitem[Carpinteri et~al.(2011)Carpinteri, Cornetti, and
  Sapora]{CarpinteriCornettiSapora-2011}
A.~Carpinteri, P.~Cornetti, and A.~Sapora.
\newblock A fractional calculus approach to nonlocal elasticity.
\newblock \emph{The European Physical Journal Special Topics}, 193\penalty0
  (1):\penalty0 193--204, 2011.

\bibitem[Ceki\'{c} et~al.(2020)Ceki\'{c}, Lin, and R\"{u}land]{CLR18}
M.~Ceki\'{c}, Y-H. Lin, and A.~R\"{u}land.
\newblock The {C}alder\'{o}n problem for the fractional {S}chr\"{o}dinger
  equation with drift.
\newblock \emph{Calc. Var. Partial Differential Equations}, 59\penalty0
  (3):\penalty0 Paper No. 91, 46, 2020.
\newblock ISSN 0944-2669.
\newblock \doi{10.1007/s00526-020-01740-6}.
\newblock URL \url{https://doi.org/10.1007/s00526-020-01740-6}.

\bibitem[Choulli and Ouhabaz(2024)]{CO23}
M.~Choulli and E.~Ouhabaz.
\newblock Fractional anisotropic calderón problem on complete riemannian
  manifolds.
\newblock \emph{Communications in Contemporary Mathematics}, 26\penalty0 (9),
  2024.

\bibitem[Covi(2020{\natexlab{a}})]{C20}
G.~Covi.
\newblock Inverse problems for a fractional conductivity equation.
\newblock \emph{Nonlinear Anal.}, 193:\penalty0 111418, 18, 2020{\natexlab{a}}.
\newblock ISSN 0362-546X.
\newblock \doi{10.1016/j.na.2019.01.008}.
\newblock URL \url{https://doi.org/10.1016/j.na.2019.01.008}.

\bibitem[Covi(2020{\natexlab{b}})]{C20a}
G.~Covi.
\newblock An inverse problem for the fractional {S}chr\"{o}dinger equation in a
  magnetic field.
\newblock \emph{Inverse Problems}, 36\penalty0 (4):\penalty0 045004, 24,
  2020{\natexlab{b}}.
\newblock ISSN 0266-5611.
\newblock \doi{10.1088/1361-6420/ab661a}.
\newblock URL \url{https://doi.org/10.1088/1361-6420/ab661a}.

\bibitem[Covi(2021)]{C21}
G.~Covi.
\newblock {Uniqueness for the fractional Calderón problem with quasilocal
  perturbations}.
\newblock \emph{Preprint, arXiv:2110.11063, to appear in SIAM Journal on
  Mathematical Analysis (SIMA)}, 2021.

\bibitem[Covi(2022)]{Cov22}
G.~Covi.
\newblock Uniqueness for the anisotropic fractional conductivity equation.
\newblock \emph{Preprint arXiv:2212.11331}, 2022.

\bibitem[Covi(2024)]{C24}
G.~Covi.
\newblock The inverse problem for the fractional conductivity equation: a
  survey.
\newblock \emph{Preprint arXiv:2408.14200}, 2024.

\bibitem[Covi and Lassas(2024)]{CL24}
G.~Covi and M.~Lassas.
\newblock An inverse problem for fractional random walks on finite graphs.
\newblock \emph{Preprint arXiv:2408.05072}, 2024.

\bibitem[Covi et~al.(2021)Covi, Mönkkönen, and Railo]{CMR21}
G.~Covi, K.~Mönkkönen, and J.~Railo.
\newblock Unique continuation property and poincarè inequality for higher
  order fractional laplacians with applications in inverse problems.
\newblock \emph{Inverse Problems \& Imaging}, 15\penalty0 (4):\penalty0
  641--681, 2021.

\bibitem[Covi et~al.(2022{\natexlab{a}})Covi, Garcia-Ferrero, and
  Rüland]{CGR22}
G.~Covi, M.~A. Garcia-Ferrero, and A.~Rüland.
\newblock On the calderón problem for nonlocal schrödinger equations with
  homogeneous, directionally antilocal principal symbols.
\newblock \emph{Journal of Differential Equations}, 341:\penalty0 79--149,
  2022{\natexlab{a}}.

\bibitem[Covi et~al.(2022{\natexlab{b}})Covi, Mönkkönen, Railo, and
  Uhlmann]{CMRU22}
G.~Covi, K.~Mönkkönen, J.~Railo, and G.~Uhlmann.
\newblock The higher order fractional calderón problem for linear local
  operators: uniqueness.
\newblock \emph{Advances in Mathematics}, 399, 2022{\natexlab{b}}.

\bibitem[Covi et~al.(2024{\natexlab{a}})Covi, de~Hoop, and Salo]{CdHS24}
G.~Covi, M.~de~Hoop, and M.~Salo.
\newblock Uniqueness in an inverse problem of fractional elasticity.
\newblock \emph{Proceedings of the Royal Society A}, 479\penalty0 (2278),
  2024{\natexlab{a}}.

\bibitem[Covi et~al.(2024{\natexlab{b}})Covi, Railo, Tyni, and
  Zimmermann]{CRTZ22}
G.~Covi, J.~Railo, T.~Tyni, and P.~Zimmermann.
\newblock Stability estimates for the inverse fractional conductivity problem.
\newblock \emph{SIAM Journal on Mathematical Analysis}, 56\penalty0 (22),
  2024{\natexlab{b}}.

\bibitem[Du et~al.(2012)Du, Gunzburger, Lehoucq, and Zhou]{DGLZ12}
Q.~Du, M.~Gunzburger, R.~B. Lehoucq, and K.~Zhou.
\newblock Analysis and approximation of nonlocal diffusion problems with volume
  constraints.
\newblock \emph{SIAM Rev.}, 54\penalty0 (4):\penalty0 667--696, 2012.
\newblock ISSN 0036-1445.
\newblock \doi{10.1137/110833294}.
\newblock URL \url{https://doi.org/10.1137/110833294}.

\bibitem[Du et~al.(2013)Du, Gunzburger, Lehoucq, and Zhou]{DGLZ13}
Q.~Du, M.~Gunzburger, R.~B. Lehoucq, and K.~Zhou.
\newblock A nonlocal vector calculus, nonlocal volume-constrained problems, and
  nonlocal balance laws.
\newblock \emph{Math. Models Methods Appl. Sci.}, 23\penalty0 (3):\penalty0
  493--540, 2013.
\newblock ISSN 0218-2025.
\newblock \doi{10.1142/S0218202512500546}.
\newblock URL \url{https://doi.org/10.1142/S0218202512500546}.

\bibitem[Eringen(1972)]{Eringen-1972}
A.~Eringen.
\newblock Linear theory of nonlocal elasticity and dispersion of plane waves.
\newblock \emph{International Journal of Engineering Science}, 10\penalty0
  (5):\penalty0 425--435, 1972.

\bibitem[Eringen(1984)]{Eri84}
A.~Eringen.
\newblock \emph{Theory of Nonlocal Elasticity and Some Applications}.
\newblock Defense Technical Information Center, Princeton., 1984.
\newblock URL \url{https://books.google.fi/books?id=Qt7stgAACAAJ}.

\bibitem[Eringen(2002)]{Er02}
A.~Eringen.
\newblock \emph{Nonlocal continuum field theories}.
\newblock Springer-Verlag, New York, 2002.
\newblock ISBN 0-387-95275-6.

\bibitem[Eringen et~al.(1977)Eringen, Speziale, and Kim]{ESK77}
A.~Eringen, C.~Speziale, and B.~Kim.
\newblock Crack-tip problem in non-local elasticity.
\newblock \emph{Journal of the Mechanics and Physics of Solids}, 25\penalty0
  (5):\penalty0 339--355, 1977.
\newblock ISSN 0022-5096.

\bibitem[Failla and Zingales(2020)]{Failla-Zingales-2020}
G.~Failla and M.~Zingales.
\newblock Advanced materials modelling via fractional calculus: {C}hallenges
  and perspectives, 2020.

\bibitem[Feizmohammadi et~al.(2024)Feizmohammadi, Krupchyk, and Uhlmann]{FKU24}
A.~Feizmohammadi, K.~Krupchyk, and G.~Uhlmann.
\newblock Calderón problem for fractional schrödinger operators on closed
  riemannian manifolds.
\newblock \emph{Preprint arXiv:2407.16866}, 2024.

\bibitem[Ghosh et~al.(2017)Ghosh, Lin, and Xiao]{GLX17}
T.~Ghosh, Y.-H. Lin, and J.~Xiao.
\newblock The calderón problem for variable coefficients nonlocal elliptic
  operators.
\newblock \emph{Communications in Partial Differential Equations}, 42\penalty0
  (12):\penalty0 1923--1961, 2017.

\bibitem[Ghosh et~al.(2020{\natexlab{a}})Ghosh, Rüland, Salo, and
  Uhlmann]{GRSU18}
T.~Ghosh, A.~Rüland, M.~Salo, and G.~Uhlmann.
\newblock Uniqueness and reconstruction for the fractional calderón problem
  with a single measurement.
\newblock \emph{Journal of Functional Analysis}, 279\penalty0 (1),
  2020{\natexlab{a}}.

\bibitem[Ghosh et~al.(2020{\natexlab{b}})Ghosh, Salo, and Uhlmann]{GSU16}
T.~Ghosh, M.~Salo, and G.~Uhlmann.
\newblock The calderón problem for the fractional schrödinger equation.
\newblock \emph{Analysis and PDEs}, 13\penalty0 (2):\penalty0 455--475,
  2020{\natexlab{b}}.

\bibitem[Griffith(1920)]{Griffith1920}
A.~Griffith.
\newblock The phenomena of rupture and flow in solids.
\newblock \emph{Philosophical Transactions, Series A}, 221:\penalty0 163--198,
  1920.

\bibitem[H{\"o}rmander(2003)]{HO:analysis-of-pdos}
L.~H{\"o}rmander.
\newblock \emph{The analysis of linear partial differential operators. {I}}.
\newblock Classics in Mathematics. Springer-Verlag, Berlin, 2003.
\newblock ISBN 3-540-00662-1.
\newblock \doi{10.1007/978-3-642-61497-2}.
\newblock URL \url{https://doi.org/10.1007/978-3-642-61497-2}.
\newblock Distribution theory and Fourier analysis, Reprint of the second
  (1990) edition [Springer, Berlin; MR1065993 (91m:35001a)].

\bibitem[Jin and Rundell(2015)]{JinRundell2015}
B.~Jin and W.~Rundell.
\newblock A tutorial on inverse problems for anomalous diffusion processes.
\newblock \emph{Inverse Problems}, 31\penalty0 (3):\penalty0 035003, 40, 2015.
\newblock ISSN 0266-5611.
\newblock \doi{10.1088/0266-5611/31/3/035003}.
\newblock URL \url{https://doi.org/10.1088/0266-5611/31/3/035003}.

\bibitem[Koch et~al.(2021)Koch, Rüland, and Salo]{KRS21}
H.~Koch, A.~Rüland, and M.~Salo.
\newblock On instability mechanisms for inverse problems.
\newblock \emph{Ars Inveniendi Analytica}, 7, 2021.

\bibitem[Kow et~al.(2022)Kow, Lin, and Wang]{KLW21}
P.-Z. Kow, Y.-H. Lin, and J.-N. Wang.
\newblock The calderón problem for the fractional wave equation: Uniqueness
  and optimal stability.
\newblock \emph{SIAM Journal of Mathematical Analysis}, 54\penalty0 (3), 2022.

\bibitem[Lai et~al.(2020)Lai, Lin, and R\"{u}land]{LLR19}
R-Y. Lai, Y-H. Lin, and A.~R\"{u}land.
\newblock The {C}alder\'{o}n problem for a space-time fractional parabolic
  equation.
\newblock \emph{SIAM J. Math. Anal.}, 52\penalty0 (3):\penalty0 2655--2688,
  2020.
\newblock ISSN 0036-1410.
\newblock \doi{10.1137/19M1270288}.
\newblock URL \url{https://doi.org/10.1137/19M1270288}.

\bibitem[Li(2020{\natexlab{a}})]{Li20a}
L.~Li.
\newblock A semilinear inverse problem for the fractional magnetic {L}aplacian.
\newblock \emph{Preprint, arXiv:2005.06714}, 2020{\natexlab{a}}.

\bibitem[Li(2020{\natexlab{b}})]{Li20b}
L.~Li.
\newblock The {C}alder\'{o}n problem for the fractional magnetic operator.
\newblock \emph{Inverse Problems}, 36\penalty0 (7):\penalty0 075003, 14,
  2020{\natexlab{b}}.
\newblock ISSN 0266-5611.
\newblock \doi{10.1088/1361-6420/ab8445}.
\newblock URL \url{https://doi.org/10.1088/1361-6420/ab8445}.

\bibitem[Lyakhovsky et~al.(1997)Lyakhovsky, Ben-Zion, and Agnon]{BenZion}
V.~Lyakhovsky, Y.~Ben-Zion, and A.~Agnon.
\newblock Distributed damage, faulting, and friction.
\newblock \emph{Journal of Geophysical Research: Solid Earth}, 102\penalty0
  (B12):\penalty0 27635--27649, 1997.

\bibitem[Ma et~al.(2022)Ma, Sahoo, and Salo]{MSS22}
S.~Ma, S.~Sahoo, and M.~Salo.
\newblock The anisotropic calderón problem at large fixed frequency on
  manifolds with invertible ray transform.
\newblock \emph{Journal of the London Mathematical Society}, 110, 07 2022.
\newblock \doi{10.1112/jlms.13006}.

\bibitem[Mindlin(1963)]{Mindlin-1964}
R.~Mindlin.
\newblock Microstructure in linear elasticity.
\newblock Technical report, Columbia Univ., New York, Dept. of Civil
  Engineering and Engineering Mechanics, 1963.

\bibitem[Mindlin(1965)]{Mindlin-1965}
R.~Mindlin.
\newblock Second gradient of strain and surface-tension in linear elasticity.
\newblock \emph{International Journal of Solids and Structures}, 1\penalty0
  (4):\penalty0 417--438, 1965.

\bibitem[Paternain et~al.(2023)Paternain, Salo, and Uhlmann]{PSU23}
G.~Paternain, M.~Salo, and G.~Uhlmann.
\newblock \emph{Geometric Inverse Problems : With Emphasis on Two Dimensions.}
\newblock Cambridge University Press. Cambridge Studies in Advanced
  Mathematics, 204, 2023.

\bibitem[Petersdorff and Stephan(1990)]{Petersdorff-Stephan}
T.~V. Petersdorff and E.~P. Stephan.
\newblock Decompositions in edge and corner singularities for the solution of
  the dirichlet problem of the laplacian in a polyhedron.
\newblock \emph{Mathematische Nachrichten}, 149\penalty0 (1):\penalty0 71--103,
  1990.

\bibitem[Railo and Zimmermann(2023{\natexlab{a}})]{RZ22b}
J.~Railo and P.~Zimmermann.
\newblock Fractional calderón problems and poincaré inequalities on unbounded
  domains.
\newblock \emph{Journal of Spectral Theory}, 13:\penalty0 63--131,
  2023{\natexlab{a}}.

\bibitem[Railo and Zimmermann(2023{\natexlab{b}})]{RZ22c}
J.~Railo and P.~Zimmermann.
\newblock Counterexamples to uniqueness in the inverse fractional conductivity
  problem with partial data.
\newblock \emph{Inverse Problems and Imaging}, 17\penalty0 (2),
  2023{\natexlab{b}}.

\bibitem[Railo and Zimmermann(2024)]{RZ22a}
J.~Railo and P.~Zimmermann.
\newblock Low regularity theory for the inverse fractional conductivity
  problem.
\newblock \emph{Nonlinear Analysis}, 239, 2024.

\bibitem[R\"{u}land and Salo(2020)]{RS17a}
A.~R\"{u}land and M.~Salo.
\newblock Quantitative approximation properties for the fractional heat
  equation.
\newblock \emph{Math. Control Relat. Fields}, 10\penalty0 (1):\penalty0 1--26,
  2020.
\newblock ISSN 2156-8472.
\newblock \doi{10.3934/mcrf.2019027}.
\newblock URL \url{https://doi.org/10.3934/mcrf.2019027}.

\bibitem[Rüland(2021)]{R20}
A.~Rüland.
\newblock On single measurement stability for the fractional calderón problem.
\newblock \emph{SIAM Journal of Mathematical Analysis}, 53\penalty0 (5), 2021.

\bibitem[Rüland and Salo(2018)]{RS-expo}
A.~Rüland and M.~Salo.
\newblock Exponential instability in the fractional calderón problem.
\newblock \emph{Inverse Problems}, 34\penalty0 (4):\penalty0 045003, February
  2018.

\bibitem[Rüland and Salo(2020)]{RS17b}
A.~Rüland and M.~Salo.
\newblock The fractional calderón problem: low regularity and stability.
\newblock \emph{Nonlinear Analysis}, 193, 2020.

\bibitem[Rüland and Sincich(2019)]{RS18}
A.~Rüland and E.~Sincich.
\newblock Lipschitz stability for the finite dimensional fractional calderón
  problem with finite cauchy data.
\newblock \emph{Inverse Problems and Imaging}, 13\penalty0 (5), 2019.

\bibitem[Salo(2017)]{S17}
M.~Salo.
\newblock {The fractional Calderón problem}.
\newblock \emph{Journées équations aux dérivées partielles}, Exp. No.
  7:\penalty0 8p, 2017.

\bibitem[Tarasov and Aifantis(2019)]{TarasovAifantis-2018}
V.~Tarasov and E.~Aifantis.
\newblock On fractional and fractal formulations of gradient linear and
  nonlinear elasticity.
\newblock \emph{Acta Mech.}, 230\penalty0 (6):\penalty0 2043--2070, 2019.
\newblock ISSN 0001-5970.
\newblock \doi{10.1007/s00707-019-2373-x}.
\newblock URL \url{https://doi.org/10.1007/s00707-019-2373-x}.

\bibitem[Valdinoci(2009)]{Val09}
E.~Valdinoci.
\newblock From the long jump random walk to the fractional laplacian.
\newblock \emph{Bol. Soc. Esp. Mat. Apl. SeMA}, 49, 2009.

\bibitem[Wang and Li(2022)]{WL22}
B.~Wang and J.~Li.
\newblock Nonlocal elastic theory for a medium with one or more rigid
  inclusions - mode iii deformation.
\newblock \emph{European Journal of Mechanics - A/Solids}, 93:\penalty0 104532,
  2022.

\bibitem[Wang and Spencer(2022)]{WS22}
W.~Wang and B.~Spencer.
\newblock Does elastic stress modify the equilibrium corner angle?
\newblock \emph{Journal of the Mechanics and Physics of Solids}, 167:\penalty0
  105003, October 2022.

\bibitem[Zorica and Oparnica(2020)]{ZoricaOparnica2020}
D.~Zorica and L.~Oparnica.
\newblock Energy dissipation for hereditary and energy conservation for
  non-local fractional wave equations.
\newblock \emph{Philos. Trans. Roy. Soc. A}, 378\penalty0 (2172):\penalty0
  20190295, 24, 2020.
\newblock ISSN 1364-503X.

\bibitem[Zworski(2012)]{Zworski_Semiclassical_Analysis}
M.~Zworski.
\newblock \emph{Semiclassical Analysis}.
\newblock Graduate studies in mathematics. American mathematical Society, 2012.

\end{thebibliography}

\end{document}